\documentclass[10pt]{amsart}

\usepackage{fullpage}
\usepackage{amssymb}
\usepackage{amsmath}
\usepackage{amsxtra}
\usepackage{amscd}
\usepackage{graphicx}
\usepackage{amsfonts}
\usepackage{pb-diagram}
\usepackage{float}
\usepackage{enumerate}

\numberwithin{equation}{section}

\newtheorem{thm}{Theorem}
\newtheorem{lem}{Lemma}
\newtheorem{cor}{Corollary}
\newtheorem{prop}{Proposition}

\newtheorem{defn}{Definition}

\newtheorem{exmp}{\it Example}

\newtheorem{rem}{\it Remark}

\def\leq{\leqslant}

\def\ul{\underline}

\def\N{{\mathbb N}}
\def\Z{{\mathbb Z}}

\def\C{{\mathbb C}}
\def\SS{{\mathfrak{T}}}

\def\YTL{{\rm YTL}_{d,n}(u)}

\def\CalZ{{\mathcal{Z}}}

\begin{document}
\title{Determination of the representations and a basis for the Yokonuma--Temperley--Lieb algebra}
\author{Maria Chlouveraki}
\address{Laboratoire de Math\'ematiques UVSQ, B\^atiment Fermat, 45 avenue des \'Etats-Unis,  78035 Versailles cedex, France.}
\email{maria.chlouveraki@uvsq.fr}
\author{Guillaume Pouchin}
\address{School of Mathematics, University of Edinburgh, JCMB King's Buildings, Edinburgh EH9 3JZ, United Kingdom.}
\email{g.pouchin@ed.ac.uk}


\thanks{We are grateful to  Sofia Lambropoulou, Jes\'us Juyumaya and Dimoklis Goundaroulis for pointing out this problem to us, and for many fruitful conversations.
We would  also like to thank Stephen Griffeth for discussing  about the Littlewood--Richardson rule, and for answering many of our questions on this subject. 
This material is based upon work supported by the National Science Foundation under Grant No. 0932078 000, while the authors were in residence at the Mathematical Science Research Institute (MSRI) in Berkeley, California, during 2013.
The second author gratefully acknowledges the financial support of EPSRC through the  grant  Ep/I02610x/1.
This research has been co-financed by the European Union (European
Social Fund - ESF) and Greek national funds through the Operational Program
``Education and Lifelong Learning" of the National Strategic Reference
Framework (NSRF) - Research Funding Program: THALIS}
\maketitle
\begin{abstract}
We determine the representations of the Yokonuma--Temperley--Lieb algebra, which is defined as a quotient of the Yokonuma--Hecke algebra by  generalising the construction of the classical Temperley--Lieb algebra. We then deduce the dimension of this algebra, and produce an explicit basis for it.
\end{abstract}

\section{Introduction}

Let $d,n \in \N$ and let $u$ be an indeterminate.
The Yokonuma--Hecke algebra ${\rm Y}_{d,n}(u)$ was originally introduced by Yokonuma \cite{yo} in the context of Chevalley groups as a generalisation of the Iwahori--Hecke algebra $\mathcal{H}_n(u)$ of type $A$. More precisely, if $u$ is taken to be a power of a prime number, then 
$\mathcal{H}_n(u)$ is obtained as 
 the centraliser  algebra associated to the permutation representation of 
${\rm GL}_n(\mathbb{F}_u)$ with respect to a Borel subgroup, while 
 ${\rm Y}_{d,n}(u)$ is obtained as the centraliser  algebra with respect to any maximal unipotent subgroup.
 Thus, the Yokonuma--Hecke algebra can be also regarded as a particular case of a unipotent Hecke algebra. 
 In recent years, the presentation of the algebra ${\rm Y}(d,n)$  has been transformed  by Juyumaya  \cite{ju1, juka, ju} to the one used in this paper.  
For $d=1$, the algebra  ${\rm Y}_{1,n}(u)$ coincides with $\mathcal{H}_n(u)$.

The Yokonuma--Hecke algebra ${\rm Y}_{d,n}(u)$ can be also viewed as
a deformation of the group algebra of the complex reflection group $G(d,1,n) \cong (\Z/d \Z) \wr \mathfrak{S}_n$, where $\mathfrak{S}_n$ denotes the symmetric group on $n$ letters, 
and as a quotient of the group algebra of the framed braid group $(\Z/d \Z) \wr B_n$, where $B_n$ is the classical braid group on $n$ strands (of type $A$).
Thanks to the latter description, the Yokonuma--Hecke algebra
 has interesting topological interpretations in the context of framed knots and links. Juyumaya and Lambropoulou used  ${\rm Y}_{d,n}(u)$ to define knot invariants for framed knots \cite{jula1, jula2}. They subsequently proved that these invariants can be extended to classical and singular knots \cite{jula3, jula4}. 

Some information on the representation theory of ${\rm Y}_{d,n}(u)$ has been obtained by Thiem in the general context of unipotent Hecke algebras \cite{thi,thi2,thi3}. More recently, Poulain d'Andecy and the first author \cite{ChPdA} have given an explicit description of the irreducible representations of  ${\rm Y}_{d,n}(u)$, which are parametrised by the $d$-partitions of $n$.  The dimension of   ${\rm Y}_{d,n}(u)$ is equal to $d^n n!$.

The classical Temperley--Lieb algebra ${\rm TL}_{n}(u)$ can be defined as the quotient of the Iwahori--Hecke algebra $\mathcal{H}_n(u)$ over a certain ideal \cite{jo}. This algebra gives rise to the famous knot invariant known as the `Jones polynomial', and has many applications in other areas of mathematics, such as statistical mechanics. It is well-known that the irreducible representations of  $\mathcal{H}_n(u)$ are parametrised by the partitions of $n$. In his paper, Jones showed that the irreducible representations of the Temperley--Lieb algebra, which are the irreducible representations of  $\mathcal{H}_n(u)$ that pass to the quotient ${\rm TL}_{n}(u)$, are parametrised by the partitions whose Young diagrams have at most two columns. Using this result, he determined that the dimension of   ${\rm TL}_{n}(u)$ is equal to the $n$-th Catalan number $C_n$.
Starting with the standard basis of $\mathcal{H}_n(u)$, whose cardinality is equal to $n!$, he managed to extract an explicit generating set for ${\rm TL}_{n}(u)$ of cardinality equal to  $C_n$, thus describing a basis for the classical Temperley--Lieb algebra \cite{jo2,jo}.

In \cite{gjkl}, Goundaroulis, Juyumaya, Kontogeorgis and Lambropoulou generalised Jones's construction to the case of  ${\rm Y}_{d,n}(u)$, and introduced the Yokonuma--Temperley--Lieb algebra ${\rm YTL}_{d,n}(u)$ with the aim of studying its topological properties (see also \cite{jula5} for a survey on the topic). Our aim in this paper is to study the algebraic properties of this new object.

Our first objective is to determine the irreducible representations of $\YTL$, that is,
determine which irreducible representations of ${\rm Y}_{d,n}(u)$ pass to the quotient ${\rm YTL}_{d,n}(u)$.
Thanks to Tits's deformation theorem, we transfer our problem to the group algebra case, studying instead the representations of the analogous quotient of $G(d,1,n)$. 
 We then transform the problem to a study of the restriction of representations from $G(d,1,n)$ to the symmetric group $\mathfrak{S}_n$. As we will see in Section $4$, this restriction is controlled by  Littlewood--Richardson coefficients, which we  introduce in the beginning of this paper. 
Applying the Littlewood--Richardson rule in our case allows us to obtain the first main result of this paper:

$ $\\  
{\bf Theorem 1.} {\em Let $n \geq 3$. The irreducible representations of  ${\rm YTL}_{d,n}(u)$ are parametrised by the $d$-partitions of $n$ whose
Young  $d$-diagrams have at most two columns in total.}\\

The dimension of an irreducible representation of  ${\rm Y}_{d,n}(u)$ parametrised by the $d$-partition $\lambda$ is equal to the number of standard $d$-tableaux of shape $\lambda$ (see \S\ref{multipartitions} for the definition of standard $d$-tableaux).  
Given the explicit description of the irreducible representations of $\YTL$ by Theorem \ref{Rd}, we obtain that  the dimension of $\YTL$ is equal to $$
\frac{d\,(nd-n+d+1)}{2}\,C_n -d\,(d-1),
$$
where $C_n$ is, as before, the $n$-th Catalan number (Proposition \ref{propdim}). Note that for $d=1$, we recover the known result for the classical Temperley--Lieb algebra  ${\rm TL}_{n}(u)$ .

The second half of this paper focuses on the construction of a basis for $\YTL$. We denote by $g_1,\ldots,g_{n-1},t_1,\ldots,t_n$ the generators of the Yokonuma--Hecke algebra ${\rm Y}_{d,n}(u)$, where $t_1,\ldots,t_n$ are the generators of $(\Z/d\Z)^n$ (recall that ${\rm Y}_{d,n}(u)$ is a quotient of $(\Z/d\Z)^n \rtimes B_n$).
Juyumaya \cite{ju} has shown that, starting with the standard basis $\mathcal{B}_{{\rm IH}}$ of the Iwahori--Hecke algebra $\mathcal{H}_n(u)$ (case $d=1$), one can obtain naturally a basis $\mathcal{B}_{{\rm YH}}$ for the Yokonuma--Hecke algebra ${\rm Y}_{d,n}(u)$ in the following way:
$$ \mathcal{B}_{{\rm YH}} = \left\{\,t_1^{r_1}\ldots t_n^{r_n} \omega \,\,|\,\,\omega \in \mathcal{B}_{{\rm IH}}\,,\,\,(r_1,\ldots,r_n) \in \{0,\ldots,d-1\}^n\,\right\}.$$
Note that the cardinality of the above set is $d^n n!$, as desired.
However, in our case, if we start with the basis $\mathcal{B}_{{\rm TL}}$ for the Temperley--Lieb algebra constructed by Jones \cite{jo2}, the set 
$$S:=  \left\{\,t_1^{r_1}\ldots t_n^{r_n} \omega \,\,|\,\,\omega \in \mathcal{B}_{{\rm TL}}\,,\,\,(r_1,\ldots,r_n) \in \{0,\ldots,d-1\}^n\,\right\}.$$
has cardinality $d^n C_n$, which is greater than the dimension of $\YTL$ unless $d=1$. So $S$ is not a basis, but simply  a spanning set for the Yokonuma--Temperley--Lieb algebra $\YTL$ for $d>1$.

 Every word in $S$ splits into a `framing part' $t_1^{r_1}\ldots t_n^{r_n}$ 
 and a  `braiding part' $\omega \in  \mathcal{B}_{{\rm TL}}$. For each fixed element in the braiding part, we describe a set of linear  dependence relations among the framing parts. 
  We use these relations to extract from $S$ a smaller spanning set $\boldsymbol{S}_{d,n}$ for $\YTL$. 
  We then show that the cardinality of $\boldsymbol{S}_{d,n}$ is equal to the dimension of  the algebra, thus obtain the second main result of this paper:

 $ $ \\  
{\bf Theorem 2.} {\em Let $n \geq 3$. The set   
$$\boldsymbol{S}_{d,n}= \left\{\,t_1^{r_1}\ldots t_n^{r_n} \omega \,\,|\,\,\omega \in \mathcal{B}_{{\rm TL}}\,,\,\,(r_1,\ldots,r_n) \in \mathcal{E}_{d,n}(\omega)\,\right\},$$
 where $\mathcal{E}_{d,n}(\omega)$ is a subset of  $\{0,\ldots,d-1\}^n$ explicitly described, with the use of (\ref{p}), by Propositions \ref{R'(1)} (for $\omega=1$) and
  \ref{spanning} (for $\omega \neq 1$), is a basis of $\YTL$.}\\

A remarkable fact, which allowed us to obtain that the cardinality of $\boldsymbol{S}_{d,n}$ is equal to the dimension of $\YTL$, is that
the cardinality of $\mathcal{E}_{d,n}(\omega)$ only depends on the \emph{weight} of $\omega$, that is, the number of distinct generators $g_i$ appearing in the word $\omega$. More precisely, if $m$ denotes the weight of $\omega$, we have
$$ | \mathcal{E}_{d,n}(\omega)|=  2^{n-m-1}d^2-(2^{n-m-1}-1)d - \delta_{m,0} (d^2-d).$$

\section{The Littlewood--Richardson coefficients}

In the first section of this paper we will introduce all the combinatorial tools that we will need in order to study the representation theory of the Yokonuma--Temperley--Lieb algebra.

\subsection{Partitions} 
A partition is a family of positive integers $\lambda=(\lambda_1,\dots,\lambda_k)$  such that $\lambda_1\geq\lambda_2\geq\dots\geq\lambda_k$. We set $|\lambda|:=\lambda_1+\dots+\lambda_k$, and we call $|\lambda|$ the {\em size} of $\lambda$. 

Let $n \in \N$. If $\lambda$ is a partition such that $|\lambda|=n$, we say that $\lambda$ is a {\em partition of } $n$.
We will denote by $\mathcal{P}(n)$  the set of all partitions of $n$. We also set $\mathcal{P}:=\bigcup_{n \geq 0}\mathcal{P}(n)$, the set of all partitions (including the empty one).

We identify partitions with their {\em Young diagrams}: the Young diagram $Y(\lambda)$ of $\lambda$ is a left-justified array of $k$ rows such that
the $j$-th row contains  $\lambda_j$ {\em nodes } for all $j=1,\dots,k$. We write $\theta=(x,y)$ for the node in row $x$ and column $y$. A node $\theta \in \lambda$ is called {\it removable} if the set of nodes obtained from $\lambda$ by removing $\theta$ is still a partition. 

\subsection{Skew shapes}
If $\nu, \lambda \in \mathcal{P}$ are such that $Y(\lambda) \subseteq Y(\nu)$, we can define the {\em skew shape} $\nu/\lambda$; we have
$Y(\nu/\lambda) = Y(\nu) \setminus Y(\lambda)$ and  $|\nu/\lambda| =|\nu|- |\lambda|$.
We will denote by $\mathcal{S}$ the set of all skew shapes.

Let  $\nu/\lambda \in \mathcal{S}$ and let $\mu$ be a partition such that $|\mu|=|\nu/\lambda|$. 
A {\em skew semistandard tableau} $T$ {\em of shape} $\nu/\lambda$ {\em and weight} $\mu$ is a way of filling the boxes of $Y(\nu/\lambda)$ with entries in $\{1,2,\ldots,|\mu|\}$ such that:
\begin{itemize}
\item $T_{i,j} < T_{i+1,j}$ (the entries strictly increase down the columns); \smallbreak
\item $T_{i,j} \leq T_{i,j+1}$ (the entries increase along the rows); \smallbreak
\item $\mu_i$ corresponds to the number of entries equal to $i$.
\end{itemize}
We will denote by ${\rm SST}(\nu/\lambda)$ the set of all skew semistandard tableaux of shape  $\nu/\lambda$ (with any possible weight).

\begin{exmp}\label{ex1} 
{\rm Some skew semistandard tableaux of shape  $(4,3,2)/(2,1)$ and weight $(3,2,1)$: }
$$T_1=\begin{array}{l} \,\,\,\,\,\,\,\,\,\,\,\,\fbox{\scriptsize{$1$}}\fbox{\scriptsize{$1$}}\\[-0.14em]
\,\,\,\,\,\,\fbox{\scriptsize{$1$}}\fbox{\scriptsize{$2$}}\\[-0.18em]
\fbox{\scriptsize{$2$}}\fbox{\scriptsize{$3$}}\end{array}
\,\,\,\,\,\,\,\,T_2=
\begin{array}{l} \,\,\,\,\,\,\,\,\,\,\,\,\fbox{\scriptsize{$1$}}\fbox{\scriptsize{$1$}}\\[-0.14em]
\,\,\,\,\,\,\fbox{\scriptsize{$2$}}\fbox{\scriptsize{$2$}}\\[-0.18em]
\fbox{\scriptsize{$1$}}\fbox{\scriptsize{$3$}}\end{array}
\,\,\,\,\,\,\,\,T_3=
\begin{array}{l} \,\,\,\,\,\,\,\,\,\,\,\,\fbox{\scriptsize{$1$}}\fbox{\scriptsize{$2$}}\\[-0.14em]
\,\,\,\,\,\,\fbox{\scriptsize{$1$}}\fbox{\scriptsize{$2$}}\\[-0.18em]
\fbox{\scriptsize{$1$}}\fbox{\scriptsize{$3$}}\end{array}
$$
\end{exmp}

\vskip .1cm

\subsection{Littlewood--Richardson tableaux} 
For the definition of  Littlewood--Richardson tableaux,  we follow \cite{Le}.
Let $\nu/\lambda, \, \nu'/\lambda' \in \mathcal{S}$. 
For $T \in {\rm SST}(\nu/\lambda)$ and $k,l \in \mathbb{N}$, we define $T^l_k$ to be the number of entries $l$ in row $k$ of $T$.  Two tableaux $T \in {\rm SST}(\nu/\lambda)$ and $T' \in {\rm SST}(\nu'/\lambda')$ are called \emph{companion tableaux} if $T^l_k=T'^{k}_l$ for every $k,l \in \mathbb{N}$. In this case, $T$ is called $\nu'/\lambda'$-\emph{dominant} (and $T'$ is $\nu/\lambda$-dominant). Note that if $T$ is $\nu'/\lambda'$-dominant, then
the  weight of $T$ is equal to $\nu'-\lambda':=(\nu'_1-\lambda_1',\,\nu'_2-\lambda_2',\,\ldots)$.
 
We shall simply say that $T$ is {\em $\lambda'$-dominant} if there exists a partition $\nu'$ such that $T$ is $\nu'/\lambda'$-dominant.

\begin{defn}
 \emph{A tableau $T \in {\rm SST}(\nu/\lambda)$ is a {\em Littlewood--Richardson tableau} if $T$ is $\emptyset$-dominant.}
\end{defn}

\begin{exmp}{\rm
In Example $\ref{ex1}$, the tableaux $T_1$ and $T_2$ are Littlewood--Richardson, whereas $T_3$ is not.}
\end{exmp}

\begin{lem}\label{firstrow}
All entries in the first row of a Littlewood--Richardson tableau are $1$. 
\end{lem}

\begin{proof}
Let $T$ be a skew semistandard tableau of shape  $\nu/\lambda \in \mathcal{S}$ and weight $\mu \in \mathcal{P}$. The tableau $T$ is Littlewood--Richardson if and only if it is a companion tableau to a semistandard tableau $T'$ of shape $\mu$. If there exists an entry $l\neq 1$ in the first row of $T$, then there exists an entry equal to $1$ in the $l$-th row of $T'$ (since $T'^{1}_l =T^l_1 \neq 0$). This contradicts the fact that the entries of $T'$ strictly increase down the columns. 
\end{proof}

\subsection{Littlewood--Richardson coefficients} Let $\lambda, \mu, \nu \in \mathcal{P}$ .
We define the {\em Littlewood--Richardson coefficient} $c^{\nu}_{\lambda,\mu}$ to be the number of  Littlewood--Richardson tableaux of shape $\nu/\lambda$ and weight $\mu$.

\begin{rem}
{\rm If $\nu/\lambda \notin \mathcal{S}$ or $|\nu/\lambda| \neq |\mu|$, then $c^{\nu}_{\lambda,\mu}=0$. Moreover,  we have
\begin{center}
$c^{\nu}_{\lambda,\emptyset}=\left\{\begin{array}{ll}
1, &\text{ if }\nu=\lambda \\
0, &\text{ otherwise }
\end{array}\right.$
\,\,\,\,\,and\,\,\,\,\,
$c^{\nu}_{\nu,\mu}=\left\{\begin{array}{ll}
1, &\text{ if }\mu=\emptyset \\
0, &\text{ otherwise. }
\end{array}\right.$
 \end{center}}
\end{rem}

The Littlewood--Richardson coefficients arise when decomposing a product of two Schur functions (for the definition of Schur functions, see, for example, \cite{Mac}) as a linear combination of other Schur functions: for $\lambda \in \mathcal{P}$, we denote by ${\rm s}_\lambda$ the corresponding Schur function. We  then have
\begin{equation}\label{LRR}
{\rm s}_\lambda {\rm s}_\mu = \sum_{\nu \in \mathcal{P}(n)} c^{\nu}_{\lambda,\mu}{\rm s}_\nu
\end{equation}
where $n=|\lambda|+|\mu|$. Equation (\ref{LRR}) is known as the ``Littlewood--Richardson rule''.
It implies, among other things, the commutativity and associativity of Littlewood--Richardson coefficients (see also  \cite{Mac}, or \cite{DaKo} for an alternative proof), that is,
\begin{equation}\label{comm}
c^{\nu}_{\lambda,\mu}=c^{\nu}_{\mu,\lambda}
\end{equation}
and
\begin{equation}\label{assoc}
\sum_{\sigma}c^{\sigma}_{\lambda,\mu}c^{\pi}_{\sigma,\nu}=\sum_{\tau} c^{\tau}_{\mu,\nu}c^{\pi}_{\lambda,\tau}.
\end{equation}

The following property of the Littlewood--Richardson coefficients, which we prove here, is crucial for the results in Section \ref{sec3}.

\begin{lem}\label{Lemma1}
Let $n\in \N$ and let $\lambda, \mu \in \mathcal{P}$ with $|\lambda| + |\mu|=n$. Set $\alpha:=\lambda_1+\mu_1$. 
Then
\begin{enumerate}
\item for all $\nu \in  \mathcal{P}(n)$ with  $\nu_1 > \alpha$, we have $c^{\nu}_{\lambda,\mu}=0$; \smallbreak
\item there exists $\nu \in  \mathcal{P}(n)$ such that $\nu_1 = \alpha$ and $c^{\nu}_{\lambda,\mu} > 0$.
\end{enumerate}
\end{lem}

\begin{proof} (1) If $\nu /\lambda \notin \mathcal{S}$, then $c^{\nu}_{\lambda,\mu}=0$. If $\nu /\lambda \in \mathcal{S}$, then the first row of $Y(\nu /\lambda)$ has $\nu_1-\lambda_1$ boxes. We have $\nu_1-\lambda_1>\alpha - \lambda_1=\mu_1$. Thus, if $T$ is a skew semistandard tableau of shape $\nu/\lambda$ and weight $\mu$, then the first row of $T$ must contain entries greater than $1$. By Lemma \ref{firstrow}, $T$ cannot be Littlewood--Richardson. Hence,
there exist no Littlewood--Richardson tableaux of shape $\nu/\lambda$ and weight $\mu$, that is, $c^{\nu}_{\lambda,\mu}=0$.\\
(2) Let $\nu$ be the partition of $n$ defined by $\nu_i:=\lambda_i+\mu_i$ for all $i \geq 1$. Then $\nu_1 = \alpha$, $\nu/\lambda \in \mathcal{S}$ and the $i$-th row of $Y(\nu /\lambda)$ has $\mu_i$ boxes. Let
$T$ be the skew semistandard tableau of shape $\nu /\lambda$ and weight $\mu$ obtained by filling every box of the $i$-th row of $Y(\nu /\lambda)$ with the entry $i$.
Then $T$ is $\mu/\emptyset$-dominant, and thus Littlewood--Richardson. We conclude that $c^{\nu}_{\lambda,\mu} > 0.$
\end{proof}

The next result shows what happens in the case where the weight of a Littlewood--Richardson tableau consists of just one row. This implies the so-called ``Pieri's rule'', which we will state in Section \ref{sec3}.

\begin{lem}\label{triv}
Let $n\in \N$ and let $\lambda, \mu \in \mathcal{P}$ with $|\lambda| + |\mu|=n$. Suppose that $\mu$ if of the form $(l)$ for some $1\leq l \leq n$. 
Let $\nu \in \mathcal{P}(n)$. We have $c^{\nu}_{\lambda,\mu} > 0$ if and only if the Young diagram of $\nu$ can be obtained from that of $\lambda$ by adding  $l$ boxes, with no two in the same column.
\end{lem}

\begin{proof}
Let $T$ be a skew  semistandard tableaux of shape  $\nu/\lambda \in \mathcal{S}$ and weight $\mu = (l)$. Then $T$ has $l$ boxes, and all entries in $T$ are equal to $1$. Thus, for $T$ to be semistandard, each column ot $T$ must have at most one box. We deduce that $Y(\nu)$ is obtained  from $Y(\lambda)$ by adding $l$ boxes, with no two in the same column.  Note that, in this case, $T$ is always Littlewood--Richardson.
\end{proof}

\subsection{Multipartitions}\label{multipartitions}
Let $d \in \N$. A {\em $d$-partition} $\lambda$, or a {\em Young $d$-diagram}, of size $n$ is a $d$-tuple of partitions such that the total number of nodes in the associated Young diagrams is equal to $n$. That is, we have $\lambda=(\lambda^{(0)},\lambda^{(1)},\dots,\lambda^{(d-1)})$ with $\lambda^{(0)},\lambda^{(1)},\dots,\lambda^{(d-1)}$ usual partitions such that $|\lambda^{(0)}|+|\lambda^{(1)}|+\dots+|\lambda^{(d-1)}|=n$.
We also say that $\lambda$ is a {\em $d$-partition of} $n$. 
We denote by $\mathcal{P}(d,n)$ the set of $d$-partitions of $n$. We have $\mathcal{P}(1,n)=\mathcal{P}(n)$.

A {\em standard $d$-tableau of shape } $\lambda \in  \mathcal{P}(d,n)$ is a way of filling the boxes of the Young $d$-diagram of $\lambda$ with the numbers $1,2,\ldots,n$ such that the entries strictly increase down the columns and along the rows.

\section{The Yokonuma--Temperley--Lieb algebra }

\subsection{The Yokonuma--Hecke algebra ${\rm Y}_{d,n}(u)$}\label{YH algebra}

Let  $d,\,n \in \N$ , $d \geq 1$. Let $u $ be an indeterminate.  The \emph{Yokonuma--Hecke algebra}, denoted by ${\rm Y}_{d,n}(u)$, is a $\C[u,u^{-1}]$-associative algebra generated by the elements
$$
 g_1, \ldots, g_{n-1}, t_1, \ldots, t_n
$$
subject to the following relations:
\begin{equation}\label{modular}
\begin{array}{ccrclcl}
\mathrm{(b}_1)& & g_ig_j & = & g_jg_i && \mbox{for all  $i,j=1,\ldots,n-1$ such that $\vert i-j\vert > 1$,}\\
\mathrm{(b}_2)& & g_ig_{i+1}g_i & = & g_{i+1}g_ig_{i+1} && \mbox{for  all  $i=1,\ldots,n-2$,}\\
\mathrm{(f}_1)& & t_i t_j & =  &  t_j t_i &&  \mbox{for all  $i,j=1,\ldots,n$,}\\
\mathrm{(f}_2)& & t_j g_i & = & g_i t_{s_i(j)} && \mbox{for  all  $i=1,\ldots,n-1$ and $j=1,\ldots,n$,}\\
\mathrm{(f}_3)& & t_j^d   & =  &  1 && \mbox{for all  $j=1,\ldots,n$,}
\end{array}
\end{equation}
where $s_i$ denotes the transposition $(i, i+1)$, together with the quadratic
relations:
\begin{equation}\label{quadr}
g_i^2 = 1 +  (u-1) \, e_{i}  +(u-1) \, e_{i} \, g_i \qquad \mbox{for all $i=1,\ldots,n-1$}
\end{equation}
 where
\begin{equation}\label{ei}
e_i :=\frac{1}{d}\sum_{s=0}^{d-1}t_i^s t_{i+1}^{-s}.
\end{equation}
It is easily verified that the elements $e_i$ are idempotents in ${\rm Y}_{d,n}(u)$.  Also, that the elements $g_i$ are invertible, with
\begin{equation}\label{invrs}
g_i^{-1} = g_i + (u^{-1} - 1)\, e_i + (u^{-1} - 1)\, e_i\, g_i.
\end{equation}

Let us denote by $\mathfrak{S}_n$ the symmetric group on $n$ letters. The group  $\mathfrak{S}_n$ is generated by the transpositions $s_1,\ldots,s_{n-1}$.
If we specialise $u$ to $1$, the defining relations (\ref{modular})--(\ref{quadr}) become the defining relations for the complex reflection group $G(d,1,n) \cong (\Z/d\Z) \wr \mathfrak{S}_n$. Thus, the algebra ${\rm Y}_{d,n}(u)$ is a deformation of the group algebra over $\C$ of  $G(d,1,n)$. 
Moreover, for $d=1$, the Yokonuma--Hecke algebra ${\rm Y}_{1,n}(u)$ coincides with the Iwahori--Hecke algebra $\mathcal{H}_n(u)$ of type $A$.
Hence, for $d=1$ and $u$ specialised to $1$, we obtain the group algebra over $\C$ of the symmetric group $\mathfrak{S}_n$.

Furthermore, the relations $({\rm b}_1)$, $({\rm b}_2)$, $({\rm f}_1)$ and $({\rm f}_2)$ are defining relations for the classical framed braid group $\mathcal{F}_n \cong \Z\wr B_n$,  
where $B_n$ is the classical braid group on $n$ strands, with the $t_j$'s being interpreted as the ``elementary framings" (framing 1 on the $j$th strand). The relations $t_j^d = 1$ mean that the framing of each braid strand is regarded modulo~$d$.
Thus, the algebra ${\rm Y}_{d,n}(u)$ arises naturally  as a quotient of the framed braid group algebra over the modular relations ~$\mathrm{(f}_3)$ and the quadratic relations~(\ref{quadr}). Moreover, relations (\ref{modular}) are defining relations for the  modular framed braid group
$\mathcal{F}_{d,n}\cong (\Z/d\Z) \wr B_n$, 
so the algebra ${\rm Y}_{d,n}(u)$ can be also seen  as a quotient of the modular framed braid group algebra over the  quadratic relations~(\ref{quadr}). 

Due to the relations (${\rm f}_1$) and (${\rm f}_2$), every word $w$ in ${\rm Y}_{d,n}(u)$ can be written in the form
$
w=t_1^{r_1}\ldots t_n^{r_n}\cdot \sigma
$, 
where $r_1,\ldots,r_n \in \{0,\ldots,d-1\}$ and $\sigma$ is a word in $g_1,\ldots,g_{n-1}$. That is, $w$ splits into the `framing part'
$t_1^{r_1}\ldots t_n^{r_n}$ and the `braiding part' $\sigma$. This is useful in the construction of a basis for ${\rm Y}_{d,n}(u)$, see Section \ref{sectionbasis}. 

\subsection{Representations of ${\rm Y}_{d,n}(u)$}

In \cite{ChPdA}, Poulain d'Andecy and the first author explicitly constructed the irreducible representations of ${\rm Y}_{d,n}(u)$ over $\C(u)$, and showed that they are parametrised by the $d$-partitions of $n$.
Let us denote by 
$${\rm Irr}({\rm Y}_{d,n}(u))=\{ \rho^{\lambda}\,|\, \lambda \in \mathcal{P}(d,n)\}$$
the set of irreducible representations of the algebra  $\C(u){\rm Y}_{d,n}(u):=\C(u) \otimes_{\C[u,u^{-1}]}{\rm Y}_{d,n}(u)$.
They  also showed that the algebra $\C(u){\rm Y}_{d,n}(u)$ is split semisimple and that the specialisation $u \mapsto  1$ induces a bijection between ${\rm Irr}({\rm Y}_{d,n}(u))$ and the set 
$${\rm Irr}(G(d,1,n))=\{ E^{\lambda}\,|\, \lambda \in \mathcal{P}(d,n)\}$$
of irreducible representations of the group $G(d,1,n)$ over $\C$.
The last result can be also independently obtained with the use of Tits's deformation theorem (see, for example, \cite[Theorem 7.4.6]{gp}).

\begin{rem}{\rm
The trivial representation is labelled by the $d$-partition $((n),\emptyset,\emptyset,\ldots,\emptyset)$, in every case.}
\end{rem}

\subsection{The Yokonuma--Temperley--Lieb algebra  ${\rm YTL}_{d,n}(u)$}

Let $n \geq 3$. The \emph{Yokonuma--Temperley--Lieb algebra}, denoted by  ${\rm YTL}_{d,n}(u)$, is defined in \cite{gjkl} as the quotient of the Yokonuma--Hecke algebra ${\rm Y}_{d,n}(u)$ by the two-sided  ideal
$$ I:= \langle\, g_ig_{i+1}g_i+g_ig_{i+1}+g_{i+1}g_i+g_i+g_{i+1}+1\,\,|\,\, i=1,2,\ldots,n-2 \,\rangle.$$
For $d=1$, the algebra ${\rm YTL}_{1,n}(u)$ coincides with the ordinary Temperley--Lieb algebra ${\rm TL}_n(u)$.  Set
$$G_{i,i+1}:=g_ig_{i+1}g_i+g_ig_{i+1}+g_{i+1}g_i+g_i+g_{i+1}+1 \,\text{ for all }\, i=1,2,\ldots,n-2.$$
Using the braid relations (\ref{modular})(${\rm b}_1$) and  (\ref{modular})(${\rm b}_2$), one can easily check that
$$G_{i,i+1}=(g_1g_2\ldots g_{n-1})^{i-1} G_{1,2} \,(g_1g_2\ldots g_{n-1})^{-(i-1)}.$$
Hence, the ideal $I$ is generated by the element $G_{1,2}$, that is, we have
$$ I = \langle\, g_1g_{2}g_1+g_1g_{2}+g_{2}g_1+g_1+g_{2}+1\,\rangle.$$

\subsection{Representations of ${\rm YTL}_{d,n}(u)$} Since ${\rm YTL}_{d,n}(u)$ is a quotient of ${\rm Y}_{d,n}(u)$,
by standard results in representation theory, we have that:
\begin{itemize}
\item The algebra $\C(u){\rm YTL}_{d,n}(u):=\C(u) \otimes_{\C[u,u^{-1}]}{\rm YTL}_{d,n}(u)$ is split semisimple.\smallbreak
\item The irreducible representations of  $\C(u){\rm YTL}_{d,n}(u)$ are in bijection with  the irreducible representations $\rho^{\lambda}$ of $\C(u){\rm Y}_{d,n}(u)$ satisfying:
\begin{equation}\label{quotient}
 \rho^{\lambda}(g_1g_{2}g_1+g_1g_{2}+g_{2}g_1+g_1+g_{2}+1)=0.
\end{equation}
\end{itemize}
Let us denote by $\mathcal{R}(d,n)$ the set of $d$-partitions $\lambda$ of $n$ for which (\ref{quotient}) is satisfied. We  denote by  
$${\rm Irr}({\rm YTL}_{d,n}(u))=\{ \varrho^{\lambda}\,|\, \lambda \in \mathcal{R}(d,n)\}$$
the set of irreducible representations of the algebra  $\C(u){\rm YTL}_{d,n}(u)$.
For every $\lambda \in \mathcal{R}(d,n)$, we have 
$$\varrho^{\lambda} \circ \pi = \rho^{\lambda},$$ where $\pi$ is the natural surjective homomorphism
from ${\rm Y}_{d,n}(u)$ onto ${\rm YTL}_{d,n}(u)$. The first aim of this paper will be to determine the set $\mathcal{R}(d,n)$.

\begin{prop}\label{1}
We have $\lambda \in \mathcal{R}(d,n)$ if and only if the trivial representation is not a direct summand of the restriction
$\mathrm{Res}_{\langle s_1,s_2\rangle}^{G(d,1,n)}(E^{\lambda})$.
\end{prop}

\begin{proof}
Let us consider the quotient of the group algebra $\C G(d,1,n)$ by the two-sided ideal 
$$ J:= \langle\, s_1s_{2}s_1+s_1s_{2}+s_{2}s_1+s_1+s_{2}+1 \,\rangle.$$
The quotient algebra $A:=\C G(d,1,n)/J$ is a split semisimple agebra over $\C$. For $u=1$, the Yokonuma--Temperley--Lieb algebra  specialises to $A$.
By Tits's deformation theorem, the specialisation $u \mapsto 1$ yields a bijection between ${\rm Irr}({\rm YTL}_{d,n}(u))$ and the set ${\rm Irr}(A)$ of irreducible representations of $A$.
Thus, we can write:
$${\rm Irr}(A)=\{ \mathcal{E}^{\lambda}\,|\, \lambda \in \mathcal{R}(d,n)\}.$$

Now, following the same reasoning as for the Yokonuma--Temperley--Lieb algebra, $\mathcal{E}^{\lambda} \in {\rm Irr}(A)$ if and only if
$$ E^{\lambda}(s_1s_{2}s_1+s_1s_{2}+s_{2}s_1+s_1+s_{2}+1)=0.$$
This equation is equivalent to
$$ \mathrm{Res}_{\langle s_1,s_2\rangle}^{G(d,1,n)}(E^{\lambda})(s_1s_{2}s_1+s_1s_{2}+s_{2}s_1+s_1+s_{2}+1)=0.$$
Now, the group $ \langle s_1,s_{2}\rangle \cong \mathfrak{S}_3$ has three irreducible representations, parametrised by the partitions $(3)$, $(2,1)$ and $(1,1,1)$.
Among them, only the trivial representation, labelled by the partition $(3)$, does not take the value $0$ on $s_1s_{2}s_1+s_1s_{2}+s_{2}s_1+s_1+s_{2}+1$, whence the desired result.
\end{proof}

\begin{rem}{\rm
An alternative proof for Proposition \ref{1} can be obtained by looking directly at the representations of the Yokonuma--Hecke algebra ${\rm Y}_{d,n}(u)$ and their restrictions to ${\rm Y}_{d,3}(u)$, with the use of the formulas obtained in \cite{ChPdA}.}
\end{rem}

Proposition \ref{1} has transformed the problem of determination of the irreducible representations of ${\rm YTL}_{d,n}(u)$ to a problem of  determination of the irreducible representations appearing in the restriction of a representation from $G(d,1,n)$ to $\mathfrak{S}_3$. For $d=1$, the restriction of an irreducible representation labelled by a partition $\lambda$  corresponds to the removal of (removable) nodes  from the Young diagram of $\lambda$. More specifically, if $\lambda$ is a partition of $n$, then  $ \mathrm{Res}_{\mathfrak{S}_{n-1}}^{\mathfrak{S}_n}(E^\lambda)$  is the direct sum  of all irreducible representations labelled by the partitions of $n-1$ whose Young diagrams are obtained  from the Young diagram of $\lambda$ by removing one node (each representation appearing with multiplicity $1$).
As a consequence, $ \mathrm{Res}_{\mathfrak{S}_k}^{\mathfrak{S}_n}(E^\lambda)$, where $k<n$, is a direct sum  (with various multiplicities) of all representations labelled by the partitions of $k$ whose Young diagrams are obtained  from the Young diagram of $\lambda$ by removing $n-k$ nodes.
 In particular, $ \mathrm{Res}_{\mathfrak{S}_3}^{\mathfrak{S}_n}(E^\lambda)$ is a direct sum of all representations labelled by the partitions of $3$ whose Young diagrams are obtained from the Young diagram of $\lambda$ by removing $n-3$ nodes. Hence,  the trivial representation is a direct summand of  $ \mathrm{Res}_{\mathfrak{S}_3}^{\mathfrak{S}_n}(E^\lambda)$
if and only if 
the Young diagram of $\lambda$ has more than two columns. This implies the following corollary of Proposition \ref{1}:

\begin{cor}\label{cor1}
We have $\lambda \in \mathcal{R}(d,n)$ if and only if all direct summands of
$\mathrm{Res}_{\mathfrak{S}_n}^{G(d,1,n)}(E^{\lambda})$ are labelled by Young $d$-diagrams with at most two columns.
\end{cor}

This in turn yields the following
well-known characterisation of the irreducible representations of the classical Temperley--Lieb algebra 
${\rm TL}_n(u)$:
\begin{cor}
We have 
\begin{equation}\label{R1}
\mathcal{R}(1,n) = \{\lambda \in \mathcal{P}(n)\,|\, \lambda_1 \leq 2\}.
\end{equation}That is,
$E^\lambda \in {\rm Irr}({\rm TL}_n(u))$  if and only if the Young diagram of $\lambda$ has at most two columns. 
\end{cor}

Unfortunately, for $d>1$, the restriction from $G(d,1,n)$ to $\mathfrak{S}_n$ is far more complicated than in the symmetric group case. As we will see in the next section, the combinatorics of the restriction are governed by the Littlewood--Richardson coefficients, and it is in general  difficult to judge which irreducible representations appear in  $ \mathrm{Res}_{\mathfrak{S}_n}^{G(d,1,n)}(E^{\lambda})$. Nevertheless,  the study of the Littlewood--Richardson coefficients in Section \ref{sec3} will allow us to obtain the answer to our problem, that is,  determine $\mathcal{R}(d,n)$ for any $d \in \N$.

\section{Restriction to $\mathfrak{S}_n$ and the Littlewood--Richardson rule}\label{sec3}

Our aim in this section will be to study the restriction of representations from $G(d,1,n)$ to $\mathfrak{S}_n$.
We will then use Corollary \ref{cor1} to determine the set $\mathcal{R}(d,n)$, and thus the irreducible representations of the Yokonuma--Temperley--Lieb algebra ${\rm YTL}_{d,n}(u)$.

\subsection{Induction, restriction and the Littlewood--Richardson coefficients}\label{indres}

The Littlewood--Richardson coefficients control the induction to $\mathfrak{S}_n$ from Young subgroups. Let $k,l \in \mathbb{N}$ be such that $k+l=n$.
Let $\lambda \in \mathcal{P}(k)$ and $\mu \in \mathcal{P}(l)$. Then the Littlewood--Richardson rule yields (see, for example, \cite[\S 3]{St}):
\begin{equation}
{\rm Ind}_{\mathfrak{S}_k \times \mathfrak{S}_l}^{\mathfrak{S}_n} (E^{\lambda} \boxtimes  E^{\mu})= \sum_{\nu \in \mathcal{P}(n)} c_{\lambda,\mu}^{\nu} E^{\nu}.
\end{equation}

More generally, if $\lambda^{(i)} \in \mathcal{P}(k_i)$ for $0 \leq i \leq d-1$ and $\sum_{i=0}^{d-1} k_i=n$, then
\begin{equation}\label{Youngsub}
 \,\,\,\,{\rm Ind}_{H}^{\mathfrak{S}_n} (E^{\lambda^{(0)}} \boxtimes  E^{\lambda^{(1)}}  \boxtimes  \cdots  \boxtimes  E^{\lambda^{(d-1)}}  )= 
 \sum_{\nu^{(i)} \in \mathcal{P}(k_0+\cdots+k_i)} c_{\lambda^{(0)},\lambda^{(1)}}^{\nu^{(1)}}c_{\nu^{(1)},\lambda^{(2)}}^{\nu^{(2)}}\ldots c_{\nu^{(d-2)},\lambda^{(d-1)}}^{\nu^{(d-1)}}E^{\nu^{(d-1)}},
  \end{equation}
  where $H:=\prod_{i=0}^{d-1} \mathfrak{S}_{k_i}$.
   Note that $\nu^{(d-1)} \in \mathcal{P}(n)$. Note also that, due to the commutativity and associativity of Littlewood--Richardson coefficients (Relations (\ref{comm}) and (\ref{assoc})), if
   $E^{\nu^{(d-1)}}$ appears with non-zero coefficient in (\ref{Youngsub}), then
 $\nu^{(d-1)}/\lambda^{(i)} \in \mathcal{S}$ \,for all $i=0,1,\ldots,d-1$.

Now let $G$ be any finite group.
In order to obtain a complete set of irreducible representations for the wreath product $G \wr \mathfrak{S}_n$, a problem originally solved by Specht \cite{Spe}, one needs to consider representations induced from wreath analogues of Young subgroups. In the case where $G$ is the cyclic group of order $d$, we have the following: Let   $\lambda=(\lambda^{(0)},\lambda^{(1)},\ldots,\lambda^{(d-1)})
 \in \mathcal{P}(d,n)$, and let us consider the irreducible representation $E^\lambda$ of $G(d,1,n) \cong (\Z/d\Z )\wr \mathfrak{S}_n \cong (\Z/d\Z )^n\rtimes \mathfrak{S}_n$. For all $i=0,1,\ldots,d-1$, set $k_i:=|\lambda^{(i)}|$. Then, by Specht's Theorem (see, for example, \cite[Theorem 4.1]{St}), we have
\begin{equation}
E^{\lambda}={\rm Ind}_{\widetilde{H}}^{G(d,1,n)}(E^{(\lambda^{(0)},\emptyset,\emptyset,\ldots,\emptyset)} \boxtimes  E^{(\emptyset,\lambda^{(1)},\emptyset,\ldots,\emptyset)}  \boxtimes  \cdots  \boxtimes  E^{(\emptyset,\emptyset,\emptyset,\ldots,\lambda^{(d-1)})}),
\end{equation}
 where $\widetilde{H}:=\prod_{i=0}^{d-1} G(d,1,k_i)$, which is naturally a subgroup of $G(d,1,n)$.
Now note that $(\Z/d\Z )^n \subseteq \widetilde{H}$. So we have $G(d,1,n)=\mathfrak{S}_n\widetilde{H}$ and $\mathfrak{S}_n \cap \widetilde{H}=H$, where  $H:=\prod_{i=0}^{d-1} \mathfrak{S}_{k_i}$.
Hence, Mackey's formula yields
\begin{equation}
{\rm Res}_{\mathfrak{S}_n}^{G(d,1,n)} (E^{\lambda})={\rm Ind}_H^{\mathfrak{S}_n} (E^{\lambda^{(0)}} \boxtimes  E^{\lambda^{(1)}}  \boxtimes  \cdots  \boxtimes  E^{\lambda^{(d-1)}}).
\end{equation}
Applying (\ref{Youngsub}) yields the following formula for the restriction of  irreducible representations from $G(d,1,n)$ to $\mathfrak{S}_n$:
\begin{equation}\label{restr}
{\rm Res}_{\mathfrak{S}_n}^{G(d,1,n)} (E^{\lambda})= \sum_{\nu^{(i)}\in \mathcal{P}(k_0+\cdots+k_i)} c_{\lambda^{(0)},\lambda^{(1)}}^{\nu^{(1)}}c_{\nu^{(1)},\lambda^{(2)}}^{\nu^{(2)}}\ldots c_{\nu^{(d-2)},\lambda^{(d-1)}}^{\nu^{(d-1)}}E^{\nu^{(d-1)}}.
  \end{equation}
  
  Let us denote by $c^\nu_\lambda$ the coefficient of $E^{\nu}$ in the above formula, for all $ \nu \in \mathcal{P}(n)$.
  Lemma \ref{Lemma1} can be then  generalised as follows:
  
  \begin{lem}\label{Lemma2}
Let $\lambda=(\lambda^{(0)},\lambda^{(1)},\ldots,\lambda^{(d-1)})
 \in \mathcal{P}(d,n)$. Set $\alpha:=\sum_{i=0}^{d-1}\lambda^{(i)}_1$. 
 Then
\begin{enumerate}
\item for all $\nu \in  \mathcal{P}(n)$ with  $\nu_1 > \alpha$, we have $c^{\nu}_{\lambda}=0$; \smallbreak
\item there exists $\nu \in  \mathcal{P}(n)$ such that $\nu_1 = \alpha$ and $c^{\nu}_{\lambda} > 0$.
\end{enumerate}
\end{lem}

\begin{proof} (1) We have $c^{\nu}_{\lambda}=c_{\lambda^{(0)},\lambda^{(1)}}^{\nu^{(1)}}c_{\nu^{(1)},\lambda^{(2)}}^{\nu^{(2)}}\ldots c_{\nu^{(d-2)},\lambda^{(d-1)}}^{\nu^{(d-1)}}$, where
$\nu^{(d-1)}=\nu$.
 If $c^{\nu}_{\lambda} \neq 0$, then, by Lemma \ref{Lemma1}(1), we must have
 $$\nu_1^{(i)} \leq    \nu_1^{(i-1)}+\lambda_1^{(i)}  \,\,\,\,\,\text{ for all } i=1,2,\ldots,d-1,$$
 where we take $\nu^{(0)}:=\lambda^{(0)}$.
 We deduce that
\begin{center}
$\nu^{(i)}_1 \leq \lambda^{(0)}_1+\cdots+\lambda^{(i)}_1$ \,\,\,\,\,\text{ for all }  $i=1,2,\ldots,d-1$.
\end{center}
In particular, for $i=d-1$, we obtain $\nu_1 \leq \alpha$.\\
(2) Set again $\nu^{(0)}:=\lambda^{(0)}$.
Following Lemma $2$(2),  we can define inductively $\nu^{(i)} \in \mathcal{P}(|\nu^{(i-1)}|+|\lambda^{(i)}|)$, 
for all $i=1,2,\ldots,d-1$, such that
$$\nu_1^{(i)}=\nu_1^{(i-1)}+\lambda_1^{(i)}  \,\,\,\,\,\text{ and }\,\,\,\,\, c_{\nu^{(i-1)},\lambda^{(i)}}^{\nu^{(i)}}>0.$$
We deduce that
$\nu^{(i)} \in \mathcal{P}(|\lambda^{(0)}|+\cdots+|\lambda^{(i)}|)$ and that
$\nu^{(i)}_1 = \lambda^{(0)}_1+\cdots+\lambda^{(i)}_1$, for all $i=1,2,\ldots,d-1$.

Set $\nu:=\nu^{(d-1)}$. We then have $\nu \in \mathcal{P}(n)$, $\nu_1 = \alpha$ 
 and
  \begin{center}
  $c^{\nu}_{\lambda}=c_{\lambda^{(0)},\lambda^{(1)}}^{\nu^{(1)}}c_{\nu^{(1)},\lambda^{(2)}}^{\nu^{(2)}}\ldots c_{\nu^{(d-2)},\lambda^{(d-1)}}^{\nu^{(d-1)}} > 0$.
  \end{center}
\end{proof}

\subsection{Determination of $\mathcal{R}(d,n)$}

In order to obtain a description of $\mathcal{R}(d,n)$, we will combine Corollary \ref{cor1} with Lemma \ref{Lemma2}.

\begin{prop}\label{prop2}
Let $\lambda=(\lambda^{(0)},\lambda^{(1)},\ldots,\lambda^{(d-1)})
 \in \mathcal{P}(d,n)$. All direct summands of $\mathrm{Res}_{\mathfrak{S}_n}^{G(d,1,n)}(E^\lambda)$ are labelled by Young $d$-diagrams with at most two columns if and only if 
$\sum_{i=0}^{d-1}\lambda^{(i)}_1 \leq 2$.
\end{prop}

\begin{proof} Set $\alpha:=\sum_{i=0}^{d-1}\lambda^{(i)}_1$. First suppose that $\alpha \leq 2$, and let $E^\nu$ be a direct summand of $\mathrm{Res}_{\mathfrak{S}_n}^{G(d,1,n)}(E^\lambda)$ for some $\nu \in  \mathcal{P}(n)$. 
By Lemma \ref{Lemma2}(1), if $\nu_1 >2 \geq \alpha$, then $c^{\nu}_{\lambda}=0$. So we must have $\nu_1 \leq 2$.

On the other hand, if $\alpha>2$, then, by Lemma \ref{Lemma2}(2),  there exists $\nu \in  \mathcal{P}(n)$ such that $\nu_1 = \alpha>2$ and $c^{\nu}_{\lambda} > 0$. 
Thus, $E^\nu$ is a direct summand of $\mathrm{Res}_{\mathfrak{S}_n}^{G(d,1,n)}(E^\lambda)$  whose Young diagram has more than two columns.
\end{proof}

Now, Proposition \ref{prop2} combined with Corollary \ref{cor1} yields:

\begin{thm}\label{Rd}
Let $n \geq 3$. We have 
\begin{equation}
\mathcal{R}(d,n) = \left\{\lambda \in \mathcal{P}(d,n)\,\,\,\left|\,\,\, \sum_{i=0}^{d-1}\lambda^{(i)}_1 \leq 2\right\}\right. .
\end{equation}
That is,
$E^\lambda \in {\rm Irr}({\rm YTL}_{d,n}(u))$  if and only if the Young $d$-diagram of $\lambda$ has at most two columns in total. 
\end{thm}

\subsection{Pieri's rule for type $G(d,1,n)$} The following result, known as ``Pieri's rule'', derives from the formulas in Subsection \ref{indres} with the use of Lemma \ref{triv}.

\begin{prop}\label{Pieri} Let $n = k + l$ where $k\geq 0$, $l\geq1$. Let $\mu \in \mathcal{P}(d,k)$ and $\lambda \in \mathcal{P}(d,n)$. Then
$E^{\lambda}$ is a direct summand of $\mathrm{Ind}_{G(d,1,k) \times \mathfrak{S}_l}^{G(d,1,n)}(E^{\mu} \boxtimes E^{(l)})$
if and only if  the Young $d$-diagram of $\lambda$ can be obtained from that of $\mu$ by adding  $l$ boxes, with no two in the same column. 
\end{prop}

For $d=1$, the above result is the classical Pieri's rule for the symmetric group (cf.~\cite[6.1.7]{gp}). For $d>1$, the proof is similar to the one for type $B_n \cong G(2,1,n)$ (cf.~\cite[6.1.9]{gp}).

The following corollary (case $l=3$) gives an alternative proof of Theorem \ref{Rd}.

\begin{cor}
 Let $n \geq 3$ and set  $k:=n-3$. Let $\lambda \in \mathcal{P}(d,n)$.
 There exists $\mu
 \in \mathcal{P}(d,k)$ such that
$E^{\lambda}$ is a direct summand of $\mathrm{Ind}_{G(d,1,k) \times \mathfrak{S}_3}^{G(d,1,n)}(E^{\mu} \boxtimes E^{(3)})$
if and only if $\sum_{i=0}^{d-1}\lambda^{(i)}_1 > 2$.
\end{cor}

\section{Dimension of the Yokonuma--Temperley--Lieb algebra}

Since the Yokonuma--Temperley--Lieb algebra ${\rm YTL}_n(u)$ is split semisimple over $\C(u)$, we must have
\begin{equation}
{\rm dim}_{\mathbb{C}(u)}(\C(u){\rm YTL}_{d,n}(u)) = \sum_{\lambda \in \mathcal{R}(d,n)} ({\rm dim}(\varrho^\lambda))^2=\sum_{\lambda \in \mathcal{R}(d,n)} ({\rm dim}(\rho^\lambda))^2= \sum_{\lambda \in \mathcal{R}(d,n)}({\rm dim}(E^\lambda))^2.
\end{equation}
We have that ${\rm dim}(E^\lambda)$ is equal to the number of standard $d$-tableaux of shape $\lambda$.

Now, for $d=1$, it is well-known that the dimension of the classical Temperley--Lieb algebra ${\rm TL}_n(u)$ is given by the {\em $n$-th Catalan number}
\begin{equation}\label{catalan}
C_n:=\frac{1}{n+1}\binom{2n}{n}=\frac{1}{n+1} \sum_{k=0}^n \binom{n}{k}^2.
\end{equation}
We will see that the $n$-th Catalan appears also in the dimension formula of the Yokonuma--Temperley--Lieb algebra  ${\rm YTL}_{d,n}(u)$ for $d>1$. 

 \begin{prop}\label{propdim}
 Let $n \geq 3$. We have
 \begin{equation}\label{dim}
 {\rm dim}_{\mathbb{C}(u)}(\C(u){\rm YTL}_{d,n}(u))=\frac{d\,(nd-n+d+1)}{2}\,C_n -d\,(d-1).
 \end{equation}
 \end{prop}
 \begin{proof}
 Following the description of the $d$-partitions in $\mathcal{R}(d,n)$ by Theorem \ref{Rd}, we have that 
 $$\mathcal{R}(d,n)= \mathcal{R}_1(d,n) \sqcup \mathcal{R}_2(d,n),$$
 where
 $$\mathcal{R}_1(d,n)=\left\{ \lambda \in \mathcal{P}(d,n) \,\left|\, \exists\, i \in \{0,1,\ldots,d-1\} \text { such that } \lambda^{(i)} \in \mathcal{R}(1,n) \text{ and }  \lambda^{(j)}=\emptyset,\,\forall\,  
 j \neq i \right\}\right.$$
and
$$\mathcal{R}_2(d,n)=\left\{ \lambda \in \mathcal{P}(d,n) \,\left|\, \exists\,  i_1\neq i_2 \in \{0,1,\ldots,d-1\} \text { such that } \lambda_1^{(i_1)}= \lambda_1^{(i_2)}=1 \text{ and }  \lambda^{(j)}=\emptyset ,\,\forall\,  
 j \neq i_1,i_2 \right\}\right. .$$
 We have
 $$ \sum_{\lambda \in \mathcal{R}_1(d,n)}({\rm dim}(E^\lambda))^2= d \, {\rm dim}_{\mathbb{C}(u)}(\C(u){\rm TL}_{n}(u))=d\,C_n.$$
 So it remains to calculate
 $$\sum_{\lambda \in \mathcal{R}_2(d,n)}({\rm dim}(E^\lambda))^2.$$
 
 Let $\lambda \in \mathcal{R}_2(d,n)$. Then there exist $ i_1 \neq i_2 \in \{0,1,\ldots,d-1\}$ such that  $\lambda_1^{(i_1)}= \lambda_1^{(i_2)}=1$ and $\lambda^{(j)}=\emptyset$, for all $j \neq i_1,i_2$. Assume that $\lambda^{(i_1)}=(1,1,\ldots,1)$ has $k$ parts, for some $k \in \{1,2,\ldots,n-1\}$. Then $\lambda^{(i_2)}=(1,1,\ldots,1)$ has $n-k$ parts, and
$${\rm dim}(E^\lambda)=\binom{n}{k}.$$
 Since there are $\binom{d}{2}$ choices for the pair $(i_1,i_2)$, we conclude that
 $$\sum_{\lambda \in \mathcal{R}_2(d,n)}({\rm dim}(E^\lambda))^2= \binom{d}{2} \, \sum_{k=1}^{n-1} \binom{n}{k}^2.$$
 Thus, we have
  $${\rm dim}_{\mathbb{C}(u)}(\C(u){\rm YTL}_{d,n}(u))=d\,C_n + \frac{d(d-1)}{2} \, \sum_{k=1}^{n-1} \binom{n}{k}^2.$$
 Due to (\ref{catalan}), we can replace the sum in the above formula by $(n+1)\,C_n-2$; this yields (\ref{dim}). 
 \end{proof}

\section{A basis for the Yokonuma--Temperley--Lieb algebra}\label{sectionbasis}

In this section we will construct an explicit basis of the Yokonuma--Temperley--Lieb algebra. In order to do this, we start from a natural spanning set, then remove as many elements as we can and check that the number of elements that remain is equal to the dimension of the algebra. For this, we will need to use  induction on the degree of words in ${\rm Y}_{d,n}(u)$: As we saw at the end of \S\ref{YH algebra}, every word $w$ in ${\rm Y}_{d,n}(u)$ splits into a `framing part'
and a  `braiding part'. We define the \emph{degree of a word} $w=t_1^{r_1}\ldots t_n^{r_n} g_{i_1}g_{i_2}\ldots g_{i_m}$ in ${\rm Y}_{d,n}(u)$ to be the integer $m$. We write ${\rm deg}(w)=m$.  
For any two words $w,w' \in  {\rm Y}_{d,n}(u)$, we have ${\rm deg}(ww')={\rm deg}(w)+{\rm deg}(w')$.

\subsection{A spanning set for $\YTL$}\label{6.1} Let $n \in \N$.
Let $\underline{i}=(i_1,\ldots,i_p)$ and $\underline{k}=(k_1, \ldots k_p)$ be two $p$-tuplets of non-negative integers, with $0 \leq p \leq n-1$.
We denote by $\mathfrak{H}_n$ the set of pairs $(\underline{i},\underline{k})$ such that
$$1\leq i_1 < i_2< \cdots < i_p \leq n-1 \,\,\,\,\,\text{and}\,\,\,\,\, i_j - k_j \geq 1\,\,\,\,\,\forall\,j=1,\ldots,p.$$
Moreover, we denote by $\SS_n$ the subset of $\mathfrak{H}_n$ consisting of the pairs $(\underline{i},\underline{k})$ such that
$$1\leq i_1 < i_2< \cdots < i_p \leq n-1 \,\,\,\,\,\text{and}\,\,\,\,\, 1\leq i_1-k_1 < i_2-k_2< \cdots < i_p-k_p \leq n-1.$$
For $(\underline{i},\underline{k}) \in \mathfrak{H}_n$, we will denote by $g_{\ul{i},\ul{k}}$ the element
$$(g_{i_1}g_{i_1-1}\ldots g_{i_1-k_1})(g_{i_2}g_{i_2-1}\ldots g_{i_2-k_2})\ldots (g_{i_p}g_{i_p-1}\ldots g_{i_p-k_p}).$$ 
For the sake of simplicity, we will write $g_{i_j,k_j}$ for the ``cycle''  $g_{i_j}g_{i_j-1}\ldots g_{i_j-k_j}$, and so we have $g_{\ul{i},\ul{k}}:=g_{i_1,k_1}g_{i_2,k_2}\ldots g_{i_p,k_p}$.
We take $g_{\emptyset,\emptyset}$ to be equal to $1$.

Let us first take $d=1$. Following Jones \cite{jo}, we consider the standard basis of the Iwahori--Hecke algebra $\mathcal{H}_n(u)$ of type $A$:
\begin{equation}\label{basisHecke}
\{g_{\ul{i},\ul{k}} \,\,|\,\,
(\underline{i},\underline{k}) \in \mathfrak{H}_n\}.
\end{equation}
The above set has $n!$ elements and it is a basis of ``reduced words'', that is, 
$$s_{\ul{i},\ul{k}}:=(s_{i_1}s_{i_1-1}\ldots s_{i_1-k_1})(s_{i_2}s_{i_2-1}\ldots s_{i_2-k_2})\ldots (s_{i_p}s_{i_p-1}\ldots s_{i_p-k_p})$$
is a reduced expression for the corresponding element in $\mathfrak{S}_n$, and every element of $\mathfrak{S}_n$ can be written in the form $s_{\ul{i},\ul{k}}$ for some $(\underline{i},\underline{k}) \in \mathfrak{H}_n$. As a consequence,
if $w$ is a word in $\mathcal{H}_n(u)$, then 
$$w \in  {\rm Span}_{\C[u,u^{-1}]}\{g_{\ul{i},\ul{k}}\ | (\underline{i},\underline{k}) \in \mathfrak{H}_n,\, {\rm deg}(g_{\ul{i},\ul{k}})\leq {\rm deg}(w) \}.$$

Now, Jones \cite{jo2} has shown that the set
\begin{equation}\label{basisTL}
\{g_{\ul{i},\ul{k}} \,\,|\,\,
(\underline{i},\underline{k}) \in \SS_n\}.
\end{equation}
is a basis of the classical Temperley--Lieb algebra ${\rm TL}_n(u)$, for $n \geq 3$. The cardinality of the above set is equal to $C_n$. If $(\underline{i},\underline{k}) \in \mathfrak{H}_n \setminus \SS_n$, then $g_{\ul{i},\ul{k}}$ breaks into a linear combination of elements of smaller degree using the relation: 
\begin{equation}\label{Therelation}
g_{i}g_{i+1}g_i = -g_ig_{i+1}- g_{i+1}g_i - g_{i} - g_{i+1}-1 \quad \text{ for all } i =1, \ldots,n-2.
\end{equation}
We deduce that, for every word $w \in {\rm TL}_n(u)$, we have
$$w \in  {\rm Span}_{\C[u,u^{-1}]}\{g_{\ul{i},\ul{k}}\ | (\underline{i},\underline{k}) \in \SS_n,\, {\rm deg}(g_{\ul{i},\ul{k}})\leq {\rm deg}(w) \}.$$

Let now $d$ be any positive integer. Juyumaya \cite{ju} has shown that the following set is a basis for the Yokonuma--Hecke algebra ${\rm Y}_{d,n}(u)$:
$$\{t_1^{r_1}\ldots t_n^{r_n }g_{\ul{i},\ul{k}} \,\,|\,\,
(\underline{i},\underline{k}) \in \mathfrak{H}_n, \,\,(r_1,\ldots,r_n) \in \{0,\ldots,d-1\}^n\,\}.
$$
The cardinality of the above set is $d^nn!$. Similarly to the classical Hecke algebra case,  if $w$ is a word in ${\rm Y}_{d,n}(u)$, then 
$$w \in  {\rm Span}_{\C[u,u^{-1}]}\{t_1^{r_1}\ldots t_n^{r_n }g_{\ul{i},\ul{k}}\,\, |\,\,
 (\underline{i},\underline{k}) \in \mathfrak{H}_n,\,\,(r_1,\ldots,r_n) \in \{0,\ldots,d-1\}^n,\,\, {\rm deg}(t_1^{r_1}\ldots t_n^{r_n }g_{\ul{i},\ul{k}})\leq {\rm deg}(w) \}.$$ 

Let us consider the Yokonuma--Temperley--Lieb algebra $\YTL$. Since $(\ref{Therelation})$ holds in $\YTL$, every element $g_{\ul{i},\ul{k}}$ with  $(\underline{i},\underline{k}) \in \mathfrak{H}_n \setminus \SS_n$ breaks into a linear combination of elements  of smaller degree, and we have that
$$
S:=\{t_1^{r_1}\ldots t_n^{r_n }g_{\ul{i},\ul{k}} \,\,|\,\,
(\underline{i},\underline{k}) \in \SS_n,\,\,(r_1,\ldots,r_n) \in \{0,\ldots,d-1\}^n\,\}
$$
is a spanning set for $\YTL$, for $n \geq 3$. The cardinality of the above set is equal to $d^n C_n$, which is not equal to the dimension of $\YTL$ unless $d=1$. So $S$ is not a basis of $\YTL$ for $d \neq 1$. However, we still have that if $w$ is any word in $\YTL$, then 
$$w \in  {\rm Span}_{\C[u,u^{-1}]}\{t_1^{r_1}\ldots t_n^{r_n }g_{\ul{i},\ul{k}}\,\, |\,\,
 (\underline{i},\underline{k}) \in \SS_n,\,\,(r_1,\ldots,r_n) \in \{0,\ldots,d-1\}^n,\,\, {\rm deg}(t_1^{r_1}\ldots t_n^{r_n }g_{\ul{i},\ul{k}})\leq {\rm deg}(w) \}.$$

\subsection{A choice of linear dependence relations} 

From now on, let $d \geq 2$.
Let $A_{d,n}$ denote the group algebra of  $(\Z/d\Z)^n$ over $\C$ and let $A_{d,n}(u):=\C[u,u^{-1}] \otimes_{\C} A_{d,n}$. The algebra $A_{d,n}(u)$ is isomorphic to the subalgebra of ${\rm Y}_{d,n}(u)$ generated by $t_1,\ldots,t_n$, but not to the subalgebra of $\YTL$ generated by the $t_j$'s. To avoid confusion, we will denote by $x_1,\ldots,x_n$ the generators of $A_{d,n}$, so that
$x_j^d=1$ and $x_ix_j=x_jx_i$  for all $i,j=1,\ldots,n$. The algebra $A_{d,n}$ has a natural basis $B_{d,n}$ over $\C$ given by  ``monomials'' in $x_1,\ldots,x_n$:
$$B_{d,n}=\{x_1^{r_1}x_2^{r_2}\ldots x_n^{r_n}\,\,|\,\,(r_1,\ldots,r_n) \in \{0,\ldots,d-1\}^n\,\}.$$ 
Thus, every element  of  $A_{d,n}$ can be written in the form $P(x_1,\ldots,x_n)$, where $P$ is a polynomial in $n$ variables with coefficients in $\C$. There is a surjective 
$\C[u,u^{-1}]$-algebra morphism $\varphi$ from $A_{d,n}(u)$ to the subalgebra of $\YTL$ generated by the $t_j$'s given by 
\begin{equation}\label{barmorphism}
\varphi: P(x_1,\ldots,x_n) \mapsto P(t_1,\ldots,t_n)=:\overline{P(x_1,\ldots,x_n)}.
\end{equation}
We have 
$$S = \{ \overline{b}\,g_{\ul{i},\ul{k}} \,\,|\,\,b \in B_{d,n},\,\,
(\underline{i},\underline{k}) \in \mathfrak{T}\}.$$

Let now $w$ be any word in $\YTL$. 
Set 
\begin{center}
$S^{ \leq w}:= \{s \in S\,\,|\,\,{\rm deg}(s) \leq {\rm deg}(w)\}$ \,\,and\,\, $S^{ < w}:= \{s \in S\,\,|\,\,{\rm deg}(s) < {\rm deg}(w)\}$.
\end{center}
We have already seen that 
\begin{equation}\label{weight}
w \in {\rm Span}_{\C[u,u^{-1}]}(S^{\leq w}).
\end{equation} 
We now  define $R(w)$ to be the set
\[
R(w)=\{P(x_1,\ldots,x_n)\in A_{d,n} \ | \ P(t_1, \ldots ,t_n)\,w \in {\rm Span}_{\C[u,u^{-1}]}(S^{<w}) \}.
\] 
It is easy to see that $R(w)$ is an ideal of the commutative algebra $A_{d,n}$. The reason we introduce this ideal is the following: 
Let $(\ul{i},\ul{k}) \in \SS_n$. If we have a proper subset $\mathcal{B}_{d,n}(g_{\ul{i},\ul{k}})$ of $B_{d,n}$
such that
$$\{b_{\ul{i},\ul{k}} + R(g_{\ul{i},\ul{k}})\,|\, b_{\ul{i},\ul{k}}  \in \mathcal{B}_{d,n}(g_{\ul{i},\ul{k}})\}$$ is a basis of the quotient space $A_{d,n}/R(g_{\ul{i},\ul{k}})$, then the set 
\[
\left\{ \overline{b}_{\ul{i},\ul{k}}\,g_{\ul{i},\ul{k}} \,\left|\, (\ul{i},\ul{k}) \in \SS_n, \ b_{\ul{i},\ul{k}}  \in \mathcal{B}_{d,n}(g_{\ul{i},\ul{k}})\right. \right\}
\]
is a spanning set of the algebra $\YTL$, of cardinality smaller than $S$.
As we will see later, this set is in fact a basis of $\YTL$, because its cardinality is equal to the dimension of the algebra.

Let us now give some properties of the ideal $R(w)$. First, note that $R(1)$ is contained in the kernel of the morphism $\varphi$  defined in (\ref{barmorphism}), since
$$R(1)= \{P(x_1,\ldots,x_n)  \in A_{d,n}\ |\  P(t_1,\ldots,t_n) =0 \}.$$

\begin{lem}\label{incl}
Let $w, w'$ be words in $\YTL$. We have $R(w) \subseteq R(ww') $.
\end{lem}

\begin{proof} If $P(x_1,\ldots, x_n) \in R(w)$, then $P(t_1,\ldots, t_n)\,w \in {\rm Span}_{\C[u,u^{-1}]}(S^{< w})$.
Let $s \in  S^{<w}$. We have
$${\rm deg}(s w')  = {\rm deg}(s ) +{\rm deg}(w' ) < {\rm deg}( w) +{\rm deg}(w' )  ={\rm deg}( ww').$$ 
By (\ref{weight}), we have 
$$sw' \in  {\rm Span}_{\C[u,u^{-1}]}(S^{\leq sw'}) \subseteq  \,{\rm Span}_{\C[u,u^{-1}]}(S^{< ww'}). $$ 
So $P(t_1,\ldots, t_n)\,ww' \in {\rm Span}_{\C[u,u^{-1}]}(S^{<ww'})$, whence $P(x_1,\ldots, x_n) \in R(ww')$.
\end{proof}

\begin{cor}\label{R(1)}
Let $w$ be any word in $\YTL$. We have $R(1) \subseteq R(w)$.
\end{cor}

Now, the symmetric group $\mathfrak{S}_n$ acts on the algebras $A_{d,n}$ and $A_{d,n}(u)$ by permutation of the $x_j$'s: if $\sigma \in \mathfrak{S}_n$ and $P (x_1,\ldots ,x_n) \in A_{d,n}(u)$, then
\[
^\sigma P(x_1,\ldots, x_n):=P(x_{\sigma(1)}, \ldots , x_{\sigma(n)}).
\]
This action is still well-defined on the subalgebra of $\YTL$ generated by the $t_j$'s, which is, as we have already seen, isomorphic to $A_{d,n}(u)/{\rm Ker} \varphi$. For this, it is enough to show that ${\rm Ker} \varphi$ is invariant under the action of the symmetric group, that is, if we have $P(x_1,\ldots,x_n) \in A_{d,n}(u)$ such that $P(t_1,\ldots,t_n)=0$, then
${}^\sigma P(t_1,\ldots, t_n)=0$ for all $\sigma \in \mathfrak{S}_n$.
Indeed, for any transposition $s_i = (i,i+1) \in \mathfrak{S}_n$, if $P(t_1,\ldots,t_n)=0$, then, due to  (\ref{modular})(${\rm f}_2$),
$${}^{s_i}P(t_1,\ldots, t_n)= g_ig_i^{-1}({}^{s_i}P(t_1,\ldots, t_n))=
g_iP(t_1,\ldots,t_n)g_i^{-1}=0.$$

Let now $w=  t_1^{r_1}\ldots t_n^{r_n}g_{i_1}\ldots g_{i_m}$ be a word in $\YTL$.  We define $\sigma_w$ to be  the permutation in  $\mathfrak{S}_n$ given by
\[
\sigma_w=s_{i_1}\ldots s_{i_m}. 
\]
Then, for any element $P (x_1,\ldots ,x_n) \in A_{d,n}(u)$, we have
\begin{equation}\label{sigmaw}
w\,P(t_1,\ldots ,t_n)= {}^{\sigma_{w}} P(t_1,\ldots,t_n)\,w.
\end{equation}

\begin{lem}\label{sum} Let $w,w'$ be words in $\YTL$.
We have ${}^{\sigma_w}R(w') \subseteq R(ww')$.
\end{lem}

\begin{proof}
 If $P(x_1,\ldots,x_n) \in R(w')$, then $P(t_1,\ldots,t_n)w' \in {\rm Span}_{\C[u,u^{-1}]}(S^{<w'})$.  Let $s \in S^{<w'}$. We have
 $${\rm deg}(ws )  = {\rm deg}(w) +{\rm deg}(s) < {\rm deg}( w) +{\rm deg}(w' )  ={\rm deg}( ww').$$ 
By (\ref{weight}), we have 
$$ws \in  {\rm Span}_{\C[u,u^{-1}]}(S^{\leq ws}) \subseteq  \,{\rm Span}_{\C[u,u^{-1}]}(S^{< ww'}). $$ 
So $wP(t_1,\ldots, t_n)w' \in {\rm Span}_{\C[u,u^{-1}]}(S^{<ww'})$.
Due to (\ref{sigmaw}), we have ${}^{\sigma_{w}} P(t_1,\ldots,t_n)ww' \in  {\rm Span}_{\C[u,u^{-1}]}(S^{<ww'})$, whence
  {${}^{\sigma_w} P(x_1,\ldots,x_n) \in R(ww')$}.
\end{proof}

We will now start by giving a  list of basic linear dependence relations in $\YTL$, which will allow us to determine some of the elements in the ideals $R(g_{\ul{i},\ul{k}})$ for $({\ul{i},\ul{k}})\in \SS_n$.

\begin{prop}
We have the following relations in the algebra $\YTL$:
\begin{enumerate}
\item $g_ig_{i+1}g_i + g_ig_{i+1}+g_{i+1}g_i+g_i+g_{i+1}+1=0$\, for any  $1\leq i <n-1$, \smallbreak
\item $(t_i-t_{i+2})(1+g_i)+(t_i-t_{i+1})(g_{i+1}+g_{i+1}g_i)=0$\, for any $1\leq i <n-1$, \smallbreak
\item $(t_{i+2}-t_i)(1+g_{i+1})+(t_{i+2}-t_{i+1})(g_i+g_ig_{i+1})=0$\, for any $1\leq i <n-1$, \smallbreak
\item $(t_{i+2}-t_{i+1})(t_{i+2}-t_{i})(g_i+1)=0$\, for any $1\leq i <n-1$, \smallbreak
\item $(t_{i-1}-t_{i+1})(t_{i-1}-t_{i})(g_i+1)=0$\, for any $1< i \leq n-1$, \smallbreak
\item $(t_{i}-t_{i+1})(t_{i+1}-t_{i+2})(t_i-t_{i+2})=0$\, for any $1 \leq i <n-1$.
\end{enumerate}
\end{prop}

\begin{proof}
The first equation
\begin{equation}\label{rel1}
g_ig_{i+1}g_i + g_ig_{i+1}+g_{i+1}g_i+g_i+g_{i+1}+1=0
\end{equation}
is a defining relation for $\YTL$.
If we compute $t_i(\ref{rel1})-(\ref{rel1})t_{i+2}$, we obtain
\begin{equation}\label{rel2}
(t_i-t_{i+2})(1+g_i)+(t_{i}-t_{i+1})(g_{i+1}+g_{i+1}g_i)=0,
\end{equation}
which is the second equation. The equation
\begin{equation}\label{rel3}
(t_{i+2}-t_i)(1+g_{i+1})+(t_{i+2}-t_{i+1})(g_i+g_ig_{i+1})=0
\end{equation}
is similarly obtained by taking $t_{i+2}(\ref{rel1})-(\ref{rel1})t_i$.

Now, if we compute $t_{i+1}(\ref{rel2})-(\ref{rel2})t_{i+2}$, we obtain
\begin{equation}\label{rel4}
(t_{i+1}-t_{i+2})(t_i-t_{i+2})(g_i+1)=0,
\end{equation}
which is the fourth equation.
On the other hand, if we take $t_{i+1}(\ref{rel3})-(\ref{rel3})t_{i}$, then we obtain
$$
(t_{i+1}-t_i)(t_{i+2}-t_{i})(g_{i+1}+1)=0.
$$
If we replace $i$ by $i-1$, we obtain
\begin{equation}\label{rel5}
(t_{i}-t_{i-1})(t_{i+1}-t_{i-1})(g_{i}+1)=0.
\end{equation}
which is the fifth equation.

Finally, if we take $t_{i}(\ref{rel4})-(\ref{rel4})t_{i+1}$, we obtain
\begin{equation}\label{rel6}
(t_{i}-t_{i+1})(t_{i+1}-t_{i+2})(t_i-t_{i+2})=0.
\end{equation}
\end{proof}

\begin{cor}
We have the following relations in the algebra $\YTL$:
\begin{enumerate}
\item $(t_{i}-t_j)(t_{i}-t_{k})(t_{j}-t_k)=0$\, for any $1 \leq i,j,k \leq n$,  \smallbreak
\item $(t_{j}-t_{i})(t_{j}-t_{i+1})(g_i+1)=0$\, for any $1\leq  i \leq n-1$ and any $1 \leq j \leq n$.
\end{enumerate}
\end{cor}
\begin{proof}
As we have already mentioned, if $P(t_1,\ldots ,t_n)=0$ for some $P(x_1,\ldots ,x_n) \in A_{d,n}$, then
${}^{\sigma}P(t_1,\ldots t_n)=0$ for all $\sigma \in \mathfrak{S}_n$. Hence, (1) is obtained directly from   (\ref{rel6}).

Now, for (2), if $j \geq i+2$, set $w:=g_{i+2} g_{i+3}\ldots g_{j-1}$. We have $wg_i=g_iw$, and thus, 
$$ (t_{i+1}-t_{j})(t_i-t_{j})(g_i+1)=w^{-1}w\, (t_{i+1}-t_{j})(t_i-t_{j})(g_i+1)=w^{-1} (\ref{rel4})\, w =0.$$
If $j \leq i-1$, set $w:=g_{i-2}g_{i-3}\ldots g_j$. We have $wg_i=g_iw$, and thus, 
$$ (t_{i}-t_{j})(t_{i+1}-t_{j})(g_i+1)=w^{-1}w \,(t_{i}-t_{j})(t_{i+1}-t_{j})(g_i+1)=w^{-1} (\ref{rel5})\, w =0.$$
For $j=i,i+1$, we have $(t_{j}-t_{i})(t_{j}-t_{i+1})=0$.
\end{proof}

We can relate all these properties to the ideals $R(w)$.

\begin{cor}\label{basicrel}
We have the following:
\begin{enumerate}
\item $R(g_ig_{i+1}g_i)=A_{d,n}$ for any $1\leq i<n-1$, \smallbreak
\item $(x_i-x_{i+1}) \in R(g_{i+1}g_i)$ for any $1\leq i <n-1$, \smallbreak
\item $(x_{i+2}-x_{i+1}) \in R(g_ig_{i+1})$ for any $1\leq i <n-1$, \smallbreak
\item $(x_j-x_{i})(x_j-x_{i+1})\in R(g_i)$ for any $1\leq i \leq n-1$ and any $1\leq j \leq n$, \smallbreak
\item $(x_i-x_j)(x_i-x_k)(x_j-x_k) \in R(1)$ for any $1\leq i,j,k \leq n$.
\end{enumerate}
\end{cor}

\begin{rem}{\rm
For $d=2$, we have $(x_i-x_j)(x_i-x_k)(x_j-x_k)=0$ for any $1\leq i,j,k \leq n$.}
\end{rem}

\begin{prop}\label{propg_1g_3}
 For all $1 \leq i , j \leq n-1$ such that $|i-j|>1$, we have 
\begin{equation}\label{g_1g_3}
(x_i+x_{i+1}-x_j-x_{j+1} )\in R(g_ig_j).
\end{equation}
\end{prop}
\begin{proof} Without loss of generality, we will prove $(\ref{g_1g_3})$ for $i=1$ and $j=3$. The proof works exactly the same for any $i,\,j$ such that $|i-j|>1$.

By Corollary \ref{basicrel}, we have
\begin{center}
$(x_3-x_1)(x_3-x_2),\,(x_4-x_1)(x_4-x_2) \in R(g_1) $ and $(x_1-x_3)(x_1-x_4),\,(x_2-x_3)(x_2-x_4)\in R(g_3)$.
\end{center}
By Lemma \ref{incl}, we have $R(g_1) \subseteq R(g_1g_3)$. Since $g_1g_3=g_3g_1$, Lemma \ref{incl} also implies that $R(g_3) \subseteq R(g_1g_3)$.
We deduce that
\begin{equation}\label{whatweneed}
(x_3-x_1)(x_3-x_2),\,(x_4-x_1)(x_4-x_2),\,(x_1-x_3)(x_1-x_4),\,(x_2-x_3)(x_2-x_4)\in R(g_1g_3).
\end{equation}

Now  let $a$ be the image of $x_1+x_2-x_3-x_4$ in the quotient  $A_{d,n}/R(g_1g_3)$. We want to show that $a=0$.
Due to (\ref{whatweneed}), in $A_{d,n}/R(g_1g_3),$ we have
\[
x_1a=x_1^2+x_1x_2-x_1x_3-x_1x_4=x_1^2+x_1x_2-x_1x_3-x_1x_4-(x_1-x_3)(x_1-x_4)=x_1x_2-x_3x_4.
\]
Similarly we obtain
\begin{equation}\label{x1a}
x_1a=x_2a=x_3a=x_4a=x_1x_2-x_3x_4.
\end{equation}

Recall that $x_k^d=1$ for all $k=1,\ldots,n$. Set $e_{k,l}:=\frac{1}{d} \sum_{s=0}^{d-1}x_k^sx_l^{-s}$ for $k,l=1,\ldots,n$. The elements $e_{k,l}$ are idempotents of $A_{d,n}$ and we have $e_{k,l}x_k = e_{k,l}x_l$.  

Using (\ref{x1a}), we obtain
$e_{1,3}a = a = e_{2,4}a$, whence $a = e_{1,3}e_{2,4}a$.
However, $e_{1,3}(x_1-x_3)=0=e_{2,4}(x_2-x_4)$, and thus $e_{1,3}e_{2,4}a=0$. We deduce that $a=0$, as desired.
\end{proof}

Now let us see what happens with $R(g_{i,k})$, where $g_{i,k}=g_{i}g_{i-1}\ldots g_{i-k}$ for some $0 \leq k < i \leq n-1$.

\begin{prop}\label{onecycle}
Let $0 \leq k < i \leq n-1$.  The ideal $R(g_{i,k})$ contains the following elements:
\begin{itemize}
\item  $(x_j-x_{i})$ for all $j \in \{i-k,\ldots,i\}$, \smallbreak
\item $(x_j-x_{i})(x_j-x_{i+1})$ for all $1\leq j\leq n$.
\end{itemize}
\end{prop}
\begin{proof}
First note that, by Lemma \ref{incl}, $R(g_i) \subseteq R(g_{i,k})$. By Corollary \ref{basicrel},   $(x_j-x_{i})(x_j-x_{i+1}) \in R(g_i)$ for all $1\leq j\leq n$.

Now, following Lemma \ref{sum} and Lemma \ref{incl}, we have 
$${}^{s_is_{i-1}\ldots s_{j+2}}R(g_{j+1}g_{j}) \subseteq  R(g_ig_{i-1}\ldots g_{j+2}g_{j+1}g_{j}) \subseteq R(g_ig_{i-1}\ldots g_{j}g_{j-1}\ldots g_{i-k})=R(g_{i,k})$$
for all $j \in \{i-k,i-k+1,\ldots,i-1\}$. By Corollary \ref{basicrel}, we have that $(x_j-x_{j+1}) \in R(g_{j+1}g_{j})$.
We deduce that  ${}^{s_is_{i-1}\ldots s_{j+2}}(x_j-x_{j+1})=(x_j-x_{j+1}) \in R(g_{i,k})$ for all $j \in \{i-k,i-k+1,\ldots,i-1\}$.
Since $ R(g_{i,k})$ is an ideal, we have $(x_j-x_i) = \sum_{l=j}^{i-1}(x_l-x_{l+1})  \in R(g_{i,k})$ for all $j \in \{i-k,i-k+1,\ldots,i-1\}$.
\end{proof}

We are now in position to prove the main result of this subsection. 
\begin{prop}\label{relations} Let $(\ul{i},\ul{k}) \in \SS_n$. The ideal $R(g_{\underline{i},\underline{k}})$ contains the following elements:
\begin{itemize}
\item $(x_j-x_{i_1})$ for all $j \in \{i_1-k_1,\ldots,i_1\}$, \smallbreak
\item $(x_j-x_{i_1})(x_j-x_{i_1+1})$ for all $1\leq j\leq n$, \smallbreak
\end{itemize}
and, for $2 \leq l \leq p$,
\begin{itemize}
\item $(x_j-x_{i_l})$  for all $j \in \{m_l,m_l+1,\ldots,i_l\}$, where $m_l:=\mathrm{max}\{ i_l-k_l, i_{l-1}+2\} $,    \smallbreak
\item $(x_{i_l}+x_{i_l+1}-x_{i_1}-x_{i_1+1})$ if $i_l> i_{l-1}+1$,    \smallbreak
\item $(x_{i_l+1}-x_{i_{l-1}+1})$  if  $i_l-k_l\leq i_{l-1}+1$.
\end{itemize}
\end{prop}
\begin{proof}
We will proceed by induction on $p$, the number of ``cycles'' in $g_{\ul{i},\ul{k}}$. 
For $p=1$, the result is given by Proposition \ref{onecycle}. Now let $p>1$, and suppose that the result holds for $p-1$.
We have 
\[
g_{\ul{i},\ul{k}}=g_{i_1,k_1}g_{i_2,k_2}\ldots g_{i_{p-1},k_{p-1}} g_{i_p,k_p}.
\]
Let $\ul{i}'=(i_1,\ldots,i_{p-1})$ and $\ul{k}'=(k_1,\ldots,k_{p-1})$. We have  $(\ul{i}',\ul{k}') \in \SS_n$  and
\[
g_{\ul{i},\ul{k}}=g_{\ul{i}',\ul{k}'}g_{i_p,k_p}.
\]
By Lemma \ref{incl}, we have $R(g_{\ul{i}',\ul{k}'}) \subseteq R(g_{\underline{i},\underline{k}})$. Thus, thanks to the induction hypothesis, we only need to see what happens for $l=p$.

By Proposition \ref{onecycle}, we have that $(x_j-x_{i_p}) \in R(g_{i_p,k_p})$ for all $j \in  \{i_p-k_p,\ldots,i_p\}$.
Following Lemma \ref{sum}, we have ${}^{s_{\ul{i}',\ul{k}'}}(x_j-x_{i_p}) \in R(g_{\underline{i},\underline{k}})$ for all $j \in  \{i_p-k_p,\ldots,i_p\}$.
Set  $m_p:=\mathrm{max}\{ i_p-k_p, i_{p-1}+2\}$. For $\{m_p,m_p+1,\ldots,i_p\}$ to be non-empty, we must have $i_p > i_{p-1}+1$.
Then, for $j \in \{m_p,m_p+1,\ldots,i_p\}$, we have ${}^{s_{\ul{i}',\ul{k}'}}(x_j-x_{i_p})=(x_j-x_{i_p})$.
So $(x_j-x_{i_p}) \in R(g_{\underline{i},\underline{k}})$ for all   $j \in \{m_p,m_p+1\ldots,i_p\}$.

If $i_p >  i_{p-1}+1$, then $g_{i_p}$ commutes with  $g_{\ul{i}',\ul{k}'}$, and we have
$$g_{\ul{i},\ul{k}}=g_{i_p} g_{\ul{i}',\ul{k}'} g_{i_p-1}\ldots g_{i_p-k_p}.$$
By Proposition \ref{propg_1g_3}, we have $(x_{i_p}+x_{i_p+1}-x_{i_1}-x_{i_1+1}) \in R(g_{i_p}g_{i_1})$, and by Lemma \ref{incl}, $R(g_{i_p}g_{i_1}) \subseteq R(g_{\underline{i},\underline{k}})$.

Finally, if $i_p-k_p \leq i_{p-1}+1$, and since $i_{p-1}+1 \leq i_p$, then 
$$g_{i_p,k_p}=g_{i_p}g_{i_p-1}\ldots g_{i_{p-1}+2}g_{i_{p-1}+1} \ldots g_{i_p-k_p},$$
and we have 
$$g_{\ul{i},\ul{k}}=g_{i_1,k_1}g_{i_2,k_2}\ldots g_{i_{p-2},k_{p-2}} g_{i_p}g_{i_p-1}\ldots g_{i_{p-1}+2} g_{i_{p-1}}g_{i_{p-1}+1} g_{i_{p-1}-1}\ldots g_{i_{p-1}-k_{p-1}}
g_{i_{p-1}} \ldots g_{i_{p}-k_p}$$
  (where $ g_{i_{p-1}} \ldots g_{i_{p}-k_p}$ is taken to be equal to $1$ if $i_p-k_p = i_{p-1}+1$).
By Corollary \ref{basicrel}, we have 
$$(x_{i_{p-1}+2}-x_{i_{p-1}+1}) \in R(g_{i_{p-1}}g_{i_{p-1}+1}),$$ and by Lemma \ref{sum}, we get
$${}^{\sigma}   (x_{i_{p-1}+2}-x_{i_{p-1}+1})
\in R(g_{i_1,k_1}g_{i_2,k_2}\ldots g_{i_{p-2},k_{p-2}} g_{i_p}g_{i_p-1}\ldots g_{i_{p-1}+2} g_{i_{p-1}}g_{i_{p-1}+1}),$$
where $\sigma:=s_{i_1,k_1}s_{i_2,k_2}\ldots s_{i_{p-2},k_{p-2}} s_{i_p}s_{i_p-1}\ldots s_{i_{p-1}+2}.$
We have 
$${}^{\sigma}   (x_{i_{p-1}+2}-x_{i_{p-1}+1}) =   x_{i_{p}+1}-x_{i_{p-1}+1},$$
and by Lemma \ref{incl},
$$R(g_{i_1,k_1}g_{i_2,k_2}\ldots g_{i_{p-2},k_{p-2}} g_{i_p}g_{i_p-1}\ldots g_{i_{p-1}+2} g_{i_{p-1}}g_{i_{p-1}+1}) \subseteq R(g_{\ul{i},\ul{k}}).$$
Thus, $(x_{i_{p}+1}-x_{i_{p-1}+1}) \in  R(g_{\ul{i},\ul{k}}).$
\end{proof}

\subsection{Quotients of the algebra $A_{d,n}$ and a new spanning set for $\YTL$}

As we mentioned in the previous subsection, our aim is to study the quotient spaces $A_{d,n}/R(g_{\ul{i},\ul{k}})$ for $(\ul{i},\ul{k})\in  \SS_n$. In this subsection, we will study the quotients $A_{d,n}/R'(g_{\ul{i},\ul{k}})$, where $R'(g_{\ul{i},\ul{k}})$ is defined as follows:
\begin{itemize}
\item $R'(1)$ is the ideal of $A_{d,n}$ generated by $(x_i-x_j)(x_i-x_k)(x_j-x_k)$ for all $1 \leq i < j < k \leq n$ ; \smallbreak
\item $R'(g_{\ul{i},\ul{k}})$ is the ideal generated by $R'(1)$ and all the elements of $R(g_{\underline{i},\underline{k}})$ given in Proposition \ref{relations}.
\end{itemize}
By Corollary \ref{basicrel}, we have $R'(1) \subseteq R(1)$. Due to
Corollary \ref{R(1)} and Proposition \ref{relations}, we have $R'(g_{\ul{i},\ul{k}}) \subseteq R(g_{\ul{i},\ul{k}})$
for all  $(\ul{i},\ul{k})\in  \SS_n$. The inverse iclusion will be established in the next subsection, when we obtain that
 ${\rm dim}_{\C}(A_{d,n}/R'(g_{\ul{i},\ul{k}}))={\rm dim}_{\C}(A_{d,n}/R(g_{\ul{i},\ul{k}}))$.

Before we proceed, we will need to introduce the following notion of ``division'' in the basis $B_{d,n}$ of $A_{d,n}$.

\begin{defn}{\rm
We will say that a  monomial $M_1=x_1^{a_1}x_2^{a_2}\ldots x_n^{a_n}\in B_{d,n}$ {\em divides} another monomial $M_2=x_1^{b_1}x_2^{b_2}\ldots x_n^{b_n}\in B_{d,n}$, and  write $M_1 | M_2$, if
$0 \leq a_i \leq b_i \leq d-1$ for all $i=1,2,\ldots,n$. We will also say that $M_2$ is {\em divisible by} $M_1$.}
\end{defn}

\begin{prop}\label{R'(1)}
We have
$${\rm dim}_{\C}(A_{d,n}/R'(1)) =(2^{n-1}-1)d^2-(2^{n-1}-2)d.$$
Set
\[
\mathcal{B}_{d,n}(1):=\{ x_i^ax_{i+1}^b\prod_{j=i+2}^n x_j^{\epsilon_j} \ | \ 1 \leq i < n, \, 1\leq a <d, \,0\leq b <d,\, \epsilon_j\in\{0, 1\}\}\cup \{ x_n^b \ | \ 0\leq b <d \},
\]
that is, $\mathcal{B}_{2,n}(1)=B_{2,n}$, while for $d>2$,
$$\mathcal{B}_{d,n}(1)=\{M \in B_{d,n}\,|\, M \text{ is not divisible by  }x_ix_k^2 \text{ for all } 1 \leq i < k \leq n \text{ such that } k-i \geq 2\\\}.$$
The set
$$\{ M + R'(1) \ | \ M \in \mathcal{B}_{d,n}(1) \}$$
is a basis of $A_{d,n}/R'(1)$.
\end{prop}

\begin{proof} 
Since the group $(\mathbb{Z}/d\mathbb{Z})^n$ is abelian and $A_{d,n}$ is the group algebra of $(\mathbb{Z}/d\mathbb{Z})^n$ over $\C$, all irreducible representations of $A_{d,n}$ are of dimension $1$. Let $\mathcal{U}_d$ denote the set of all $d$-th roots of unity.
 We have a bijection
from the set ${\rm Irr}(A_{d,n})$ of irreducible representations of $A_{d,n}$ to $\mathcal{U}_d^n$
given by $\rho \mapsto (\rho(x_1),\rho(x_2),\ldots,\rho(x_n))$. 
From now on, we will identify ${\rm Irr}(A_{d,n})$ with $\mathcal{U}_d^n$.
Note that since the algebra $A_{d,n}$ is semisimple, so are its quotients.
We deduce that
$${\rm dim}_{\C}(A_{d,n}/R'(1)) = |\{(z_1,z_2,\ldots,z_n) \in \mathcal{U}_d^n\,|\,(z_i-z_j)(z_i-z_k)(z_j-z_k)=0  \text{ for all } 1 \leq i < j <k \leq n\}|.$$
This in turn is equal to the number of functions from $\{1,2,\ldots,n\}$ to $\{1,2,\ldots,d\}$ whose image has at most $2$ elements. We thus obtain
$${\rm dim}_{\C}(A_{d,n}/R'(1)) = (2^n-2) \binom{d}{2} + d= 
(2^{n-1}-1)d^2-(2^{n-1}-2)d.$$

We now have
\[
|\mathcal{B}_{d,n}(1)|=d+\sum_{i=1}^{n-1} d(d-1)2^{n-i-1}=d+ d(d-1)(2^{n-1}-1)=(2^{n-1}-1)d^2-(2^{n-1}-2)d={\rm dim}_{\C}(A_{d,n}/R'(1)).
\]
So it is enough to show that 
$$
\{ M + R'(1) \ | \ M \in \mathcal{B}_{d,n}(1) \}
$$
is a spanning set for  $A_{d,n}/R'(1)$.

For $d=2$, we have $R'(1) = \{0\}$ and $\mathcal{B}_{2,n}(1)=B_{2,n}$, the basis of $A_{2,n}$. 
From now on, we assume that $d>2$ and we identify the elements of $A_{d,n}$ with their images in $A_{d,n}/R'(1)$.
Let $M \in B_{d,n}$. 
Set 
$$i_M := {\rm min} \{ i \ | \ x_i \text{ divides } M \} \quad \text{ and } \quad    k_M := {\rm max} \{ k \ | \ x_k^2 \text{ divides } M \}.$$
If the first set is empty, then $M=1 \in \mathcal{B}_{d,n}(1)$, while if the second set is empty, then again $M \in \mathcal{B}_{d,n}(1)$.
We now assume that $i_M, k_M \in \{1,2,\ldots,n\}$. We obviously have $i_M \leq k_M$. 
We will  show by induction on the difference $k_M - i_M$ that $M$ can be written as a linear combination of elements in $\mathcal{B}_{d,n}(1)$.

If $k_M - i_M=0$, then $M$ is of the form $x_{i_M}^a \prod_{j=i_M+1}^nx_j^{\epsilon_j} $ where $2 \leq a <d$ and  $\epsilon_j\in\{0, 1\}$. By definition, $M \in \mathcal{B}_{d,n}(1)$.
If $k_M - i_M=1$, then $M$ is of the form $x_{i_M}^ax_{i_M+1}^b \prod_{j=i_M+2}^nx_j^{\epsilon_j} $ where $1 \leq a <d$, $2 \leq b <d$ and  $\epsilon_j\in\{0, 1\}$. Again, $M \in \mathcal{B}_{d,n}(1)$.
Now let $k_M - i_M \geq 2$ and assume that all monomials $M' \in  B_{d,n}$ with $k_{M'} - i_{M'}<k_M - i_M$ can be  written as linear combinations of elements in $\mathcal{B}_{d,n}(1)$.
We have that $x_{i_M}x_{k_M}^2$ divides $M$ and that
$$(x_{i_M} - x_{i_M+1})(x_{i_M} - x_{k_M}) (x_{i_M+1}-x_{k_M})=0,$$
whence
\begin{equation}\label{6terms}
x_{i_M}x_{k_M}^2=x_{i_M+1}x_{k_M}^2+x_{i_M}x_{i_M+1}^2+x_{i_M}^2x_{k_M}-x_{i_M+1}^2x_{k_M}-x_{i_M}^2x_{i_M+1}.
\end{equation}
Starting with $M$, we replace $x_{i_M}x_{k_M}^2$ by the expression above, and we obtain
a description of $M$ as linear combination of 5 monomials that might be divisible by $x_{i_M}x_{k_M}^2$ or not. If some of them are, then we replace again 
$x_{i_M}x_{k_M}^2$ with the use of (\ref{6terms}), and we go on until $M$ is a expressed as a linear combination of monomials $M'$ that are not divisible by $x_{i_M}x_{k_M}^2$ (the process eventually ends, as the sum of the power of  $x_{i_M}$ and twice the power of $x_{k_M}$ decreases at every step).
Since the only new generator appearing in the decomposition given by (\ref{6terms}) is $x_{i_M+1}$, we must have $k_{M'}  \leq k_M$ and $i_{M'} \geq i_M$ for each such $M'$, with one of the two inequalities being necessarily strict.
In every case, we have $k_{M'} - i_{M'}<k_M - i_M$ and the induction hypothesis yields the desired result.
\end{proof}

Next we want to study  the quotient $A_{d,n}/R'(g_{\ul{i},\ul{k}})$ for any $(\ul{i},\ul{k})\in \SS_n$. For this, we will introduce some notation.

\begin{defn}{\rm Let $(\ul{i},\ul{k})\in \SS_n$.
We denote by $\mathcal{I}(g_{\underline{i},\underline{k}})$ the set of all indices of the $g_j$'s appearing in $g_{\underline{i},\underline{k}}$, \emph{i.e.,}
$$\mathcal{I}(g_{\underline{i},\underline{k}})=\{i_1,i_1-1,\ldots,i_1-k_1,i_2,i_2-1,\ldots,i_2-k_2,\ldots, i_p,i_p-1,\ldots,i_p-k_p\}.$$
We define the \emph{weight} of $g_{\underline{i},\underline{k}}$ to be
 $wt(g_{\underline{i},\underline{k}}):=|\mathcal{I}(g_{\underline{i},\underline{k}})|$.
 For $j \in  \mathcal{I}(g_{\underline{i},\underline{k}})$, we set
$$ l_j := {\rm min}\{ 1 \leq l \leq p \ | \ i_l-k_l \leq j \leq i_l \}.$$
}\end{defn}

\begin{exmp}{\rm
Let $n=5$ and $g_{\ul{i},\ul{k}} = g_2g_1g_4g_3g_2$. Then
$wt(g_{\underline{i},\underline{k}})=4$, $l_1=l_2=1$ and $l_3=l_4=2$.}
\end{exmp}

\begin{prop}\label{above}
Let $(\ul{i},\ul{k})\in \SS_n \setminus \{(\emptyset,\emptyset)\}$ and let
$j\in \mathcal{I}(g_{\ul{i},\ul{k}})$. In the quotient $A_{d,n}/R'(g_{\ul{i},\ul{k}})$, $x_j$ can be expressed as a linear combination of 
 $x_{i_1}$ and $(x_s)_{s \notin \mathcal{I}(g_{\underline{i},\underline{k}})}$.
\end{prop}
\begin{proof} We identify the elements of $A_{d,n}$ with their images in $A_{d,n}/R'(g_{\ul{i},\ul{k}})$.
If $l_j=1$, then, by Proposition \ref{relations}, we have $x_j=x_{i_1}$ in the quotient $A_{d,n}/R'(g_{\ul{i},\ul{k}})$.
Now suppose $l_j>1$.  Then $j > i_{l_j-1}$.
We will have to distinguish two cases.

If $j>i_{l_j-1}+1$, then, by Proposition \ref{relations}, we have $x_j=x_{i_{l_j}}$.
Moreover, since $i_{l_j} \geq j >i_{l_j-1}+1$, we have
$x_{i_{l_j}}=x_{i_1}+x_{i_1+1}-x_{i_{l_j}+1}$.

If  $j=i_{l_j-1}+1$, then $ i_{l_j-1}+1 \geq i_{l_j}-k_{l_j}$ and, by Proposition \ref{relations}, we have $x_j=x_{i_{l_j}+1}$.

Thus, if  $j\in \mathcal{I}(g_{\ul{i},\ul{k}})$, then $x_j$ can be always expressed as a linear combination of $x_{i_1}$, $x_{i_1+1}$ and $x_{i_{l_j}+1}$.
In particular, if $j = i_{l-1}+1 \in \mathcal{I}(g_{\ul{i},\ul{k}})$ for some $2 \leq l \leq p$, then $l_j=l$ and $x_j=x_{i_{l}+1}$. If now $i_{l}+1 \in \mathcal{I}(g_{\ul{i},\ul{k}})$, we can repeat the same argument and obtain
$x_{i_{l}+1}=x_{i_{l+1}+1}$. We can go on until we find the first $l \leq L \leq p$ such that $i_{L}+1 \notin \mathcal{I}(g_{\ul{i},\ul{k}})$ (note that $i_p+1 \notin \mathcal{I}(g_{\ul{i},\ul{k}})$, so the process will eventually stop).  We then have $x_j = x_{i_L+1}$.

For $1 \leq l \leq p$, set 
\begin{equation}\label{L(l)}
L(l) : = {\rm min}\{l \leq L \leq p\,|\, i_{L}+1 \notin \mathcal{I}(g_{\underline{i},\underline{k}})\}.
\end{equation}
Following the discussion in the above paragraph, we have $x_{i_1+1}=x_{i_{L(1)}+1}$ and $x_{i_{l_j}+1}=x_{i_{L(l_j)+1}}$ for any $j \in  \mathcal{I}(g_{\ul{i},\ul{k}})$.
We conclude that if $j\in \mathcal{I}(g_{\ul{i},\ul{k}})$, then $x_j$ can be expressed as a linear combination of  $x_{i_1}$ and $(x_s)_{s \notin \mathcal{I}(g_{\underline{i},\underline{k}})}$.
\end{proof}

We are now ready to state and prove the second main result of this subsection.

 \begin{prop}\label{spanning}
 Let $(\ul{i},\ul{k})\in \SS_n \setminus \{(\emptyset,\emptyset)\}$ and let $m:=wt(g_{\underline{i},\underline{k}})$. 
 Let $L(1)$ be defined as in $(\ref{L(l)})$. We set
 $$ \widetilde{x}_1:=x_{i_1},\,\,\widetilde{x}_2:=x_{i_{L(1)}+1}\,\,\text{ and }\,\,\{\widetilde{x}_3,\ldots, \widetilde{x}_{n-m+1}\}=\{ x_s\,|\, s \notin  \mathcal{I}(g_{\underline{i},\underline{k}}),\, s \neq i_{L(1)}+1\}.$$
Set $N:=n-m+1$. Let
 $$\widetilde{B}_{d,N}:=\{ \widetilde{x}_1^{r_1} \widetilde{x}_2^{r_2}\ldots  \widetilde{x}_{N}^{r_{N}}\,\,|\,\,(r_1,\ldots,r_N) \in \{0,\ldots,d-1\}^N\,\}$$
 and 
$$\widetilde{\mathcal{B}}_{d,N}(1)=\{ \widetilde{x}_i^a\widetilde{x}_{i+1}^b\prod_{j=i+2}^{N} \widetilde{x}_j^{\epsilon_j} \ | \ 1 \leq i < N, \, 1\leq a <d, \,0\leq b <d,\, \epsilon_j\in\{0, 1\}\}\cup \{ \widetilde{x}_{N}^b \ | \ 0\leq b <d \}$$
 (as in Proposition \ref{R'(1)}). Set
$$
\mathcal{B}_{d,n}(g_{\underline{i},\underline{k}})= \{ \widetilde{x}_i^a\widetilde{x}_{i+1}^b\prod_{j=i+2}^{N} \widetilde{x}_j^{\epsilon_j} \ | \ 2 \leq i < N, \, 1\leq a <d, \,0\leq b <d,\, \epsilon_j\in\{0, 1\}\} $$ $$\cup \,\{ \widetilde{x}_{N}^b \ | \ 0\leq b <d \}\cup\{ \widetilde{x}_1^a\widetilde{x}_{2}^b \ | \ 1\leq a <d,\, 0\leq b <d \},
$$
 that is, 
$$\mathcal{B}_{d,n}(g_{\underline{i},\underline{k}}):=\widetilde{\mathcal{B}}_{d,N}(1) \cap \{M \in \widetilde{B}_{d,N}\,|\,M \text{ is not divisible by }   \widetilde{x}_{1}\widetilde{x}_j \text{ for all } j >2\}.$$
 The set
$$\{ M + R'(g_{\underline{i},\underline{k}}) \ | \ M \in \mathcal{B}_{d,n}(g_{\underline{i},\underline{k}})\}$$
is a spanning set for $A_{d,n}/R'(g_{\underline{i},\underline{k}})$.
  \end{prop}

\begin{proof} Again we  identify the elements of $A_{d,n}$ with their images in $A_{d,n}/R'(g_{\ul{i},\ul{k}})$. 
First, thanks to Proposition \ref{above}, we have that
$$\{ M + R'(g_{\underline{i},\underline{k}}) \ | \ M \in \widetilde{B}_{d,N} \}$$
 is a spanning set for $A_{d,n}/R'(g_{\ul{i},\ul{k}})$. Moreover, as we saw in the proof of Proposition \ref{above}, $x_{i_1+1}=x_{i_{L(1)}+1}$, and  following Proposition \ref{relations}, we have
\begin{equation}\label{covered}
(\widetilde{x}_j-\widetilde{x}_{1})(\widetilde{x}_j-\widetilde{x}_2)=0\,\,\,\,\,\text{for all }  1 \leq j \leq N.
\end{equation}
 Thus, for $j > 2$, we obtain
\begin{equation}\label{product}
\widetilde{x}_1\widetilde{x}_j  = \widetilde{x}_j^2 - \widetilde{x}_2\widetilde{x}_j+\widetilde{x}_1\widetilde{x}_2.
\end{equation}
Using (\ref{product}), we can remove all monomials that are divisible by $\widetilde{x}_1\widetilde{x}_j$ for $j >2$ from our spanning set.
 For $d=2$, there is nothing else to do.
 
 Let $d>2$.
Since $R'(1) \subseteq R'(g_{\ul{i},\ul{k}})$, we can
now repeat the same method as in the proof of Proposition \ref{R'(1)}, and use the relations
\begin{equation}
(\widetilde{x}_j-\widetilde{x}_{i})(\widetilde{x}_j-\widetilde{x}_{i+1})(\widetilde{x}_i-\widetilde{x}_{i+1})=0\,\,\,\,\,\text{for }  j-i \geq 2,
\end{equation}
to remove all monomials that are divisible by $\widetilde{x}_i\widetilde{x}_j^2$ for $j-i \geq 2$ from our spanning set.
Note that the case $i=1$ is already covered by (\ref{covered}).

We deduce that 
$$\{ M + R'(g_{\underline{i},\underline{k}}) \ | \ M \in\mathcal{B}_{d,n}(g_{\underline{i},\underline{k}}) \}$$
is a spanning set for $A_{d,n}/R'(g_{\underline{i},\underline{k}})$.
\end{proof}

 \begin{cor}\label{dimcor}
  Let $(\ul{i},\ul{k})\in \SS_n \setminus \{(\emptyset,\emptyset)\}$ and let $m:=wt(g_{\underline{i},\underline{k}})$.   
 We have
 $$ | \mathcal{B}_{d,n}(g_{\underline{i},\underline{k}})|=  2^{n-m-1}d^2-(2^{n-m-1}-1)d\,\geq \,{\rm dim}_{\C}(A_{d,n}/R'(g_{\underline{i},\underline{k}})).$$
 \end{cor}
 \begin{proof} Using the first expression for $\mathcal{B}_{d,n}(g_{\underline{i},\underline{k}})$ given by Proposition \ref{spanning}, we obtain
 \begin{center}
 $ |\mathcal{B}_{d,n}(g_{\underline{i},\underline{k}})|= |\mathcal{B}_{d,n-m}(1)| + d(d-1)=(2^{n-m-1}-1)d^2-(2^{n-m-1}-2)d+d^2-d=2^{n-m-1}d^2-(2^{n-m-1}-1)d.$
\end{center} 
Since $\mathcal{B}_{d,n}(g_{\underline{i},\underline{k}})$ provides a spanning set for $A_{d,n}/R'(g_{\underline{i},\underline{k}})$, we have
$ | \mathcal{B}_{d,n}(g_{\underline{i},\underline{k}})|\,\geq \,{\rm dim}_{\C}(A_{d,n}/R'(g_{\underline{i},\underline{k}})).$
\end{proof}

  In Propositions \ref{R'(1)} and \ref{spanning} we have given spanning sets for the quotient spaces $A_{d,n}/R'(g_{\underline{i},\underline{k}})$ for all $(\ul{i},\ul{k})\in \SS_n$.
  Since $R'(g_{\underline{i},\underline{k}}) \subseteq R(g_{\underline{i},\underline{k}})$, for all $(\ul{i},\ul{k})\in \SS_n$, we have that
  $$\{b_{\ul{i},\ul{k}} + R(g_{\ul{i},\ul{k}})\,|\, b_{\ul{i},\ul{k}}  \in \mathcal{B}_{d,n}(g_{\ul{i},\ul{k}})\}$$
is a spanning set for $A_{d,n}/R(g_{\underline{i},\underline{k}})$.
We deduce that
\[
\boldsymbol{S}_{d,n}:=\left\{ \overline{b}_{\ul{i},\ul{k}}\,g_{\ul{i},\ul{k}} \,\left|\, (\ul{i},\ul{k}) \in \SS_n, \ b_{\ul{i},\ul{k}}  \in \mathcal{B}_{d,n}(g_{\ul{i},\ul{k}})\right. \right\}
\]
is a spanning set for $\YTL$, where $\overline{x_1^{r_1}\ldots x_n^{r_n}}=t_1^{r_1}\ldots t_n^{r_n}$ for all $r_1,\ldots, r_n \in \{0,\ldots,d-1\}$.

\subsection{A basis for $\YTL$}
 We will now compute the number of elements of the spanning set $\boldsymbol{S}_{d,n}$ obtained in the previous subsection, and prove that it is equal to ${\rm dim}_{\C(u)}(\C(u)\YTL)$.
For $0\leq m\leq n-1$, we set
\[
\CalZ_n(m):=\{g_{\underline{i},\underline{k}}\  | \ (\underline{i},\underline{k})\in \SS_n, \  wt(g_{\underline{i},\underline{k}})=m \}\quad\text{ and } \quad Z_n(m):=|\CalZ_n(m)|.
\]
We have already seen, in  (\ref{basisTL}), that the cardinality of the set $\{g_{\ul{i},\ul{k}} \,\,|\,\,
(\underline{i},\underline{k}) \in \SS_n\}$
 is equal to the $n$-th Catalan number $C_n$. Thus, we have
\begin{equation}\label{ind1}
\sum_{m=0}^{n-1} Z_n(m)=C_n.
\end{equation}
For $n'<n$, we will view the set of indices $\SS_{n'}$ as a subset of $\SS_n$. 

\begin{lem} For any $n\geq 1$, we have the following equalities :
\begin{equation}\label{ind2}
Z_n(m)=\sum_{j=n-m-1}^{n-1} Z_j(m-n+j+1)Z_{n-j}(n-j-1) \ \  \text{  for all }\,\, 0\leq m\leq n-2, 
\end{equation}
\begin{equation}\label{lem2}
Z_n(n-1)=C_{n-1},
\end{equation}
\begin{equation}\label{lem3}
\sum_{m=0}^{n-1} 2^{n-m}Z_n(m)=(n+1)C_n.
\end{equation}
\end{lem}
\begin{proof}
We first prove (\ref{ind2}). Let 
$$g_{\underline{i},\underline{k}}=(g_{i_1}g_{i_1-1}\ldots g_{i_1-k_1})(g_{i_2}g_{i_2-1}\ldots g_{i_2-k_2})\ldots (g_{i_p}g_{i_p-1}\ldots g_{i_p-k_p}) \in \CalZ_n(m).$$
Since $wt(g_{\underline{i},\underline{k}})=m\leq n-2$, there exists at least one $1 \leq j \leq n-1$ that does not belong to $\mathcal{I}(g_{\underline{i},\underline{k}})$. 
Then there exists $l \in \{0,\ldots,p\}$ such that $i_l <j< i _{l+1} - k_{l+1}$, taking $i_0:=0$ and $ i _{p+1} - k_{p+1}:=n$. 
Set
$$g_{\underline{i}',\underline{k}'}^{(j)}:=(g_{i_1}g_{i_1-1}\ldots g_{i_1-k_1})\ldots (g_{i_l}g_{i_l-1}\ldots g_{i_l-k_l})$$ 
and
$$g_{\underline{i}'',\underline{k}''}^{(j)}:=(g_{i_{l+1}}g_{i_{l+1}-1}\ldots g_{i_{l+1}-k_{l+1}})\ldots (g_{i_p}g_{i_p-1}\ldots g_{i_p-k_p}).$$
We have
$$g_{\underline{i},\underline{k}}=g_{\underline{i}',\underline{k}'}^{(j)}g_{\underline{i}'',\underline{k}''}^{(j)}.$$
All indices in $g_{\underline{i}',\underline{k}'}^{(j)}$ are strictly less than $j$ and we have $({\underline{i}',\underline{k}'}) \in \SS_j$.
All indices in $g_{\underline{i}'',\underline{k}''}^{(j)}$ are strictly greater than $j$ and we have $({\underline{i}'',\underline{k}''})\in \SS_{n-j}$ (after relabelling the generators).
Set $$J : = {\rm max}\,\{1 \leq j \leq n-1\,|\, j \notin \mathcal{I}(g_{\underline{i},\underline{k}})\}.$$
Since $|\{1 \leq j \leq n-1\,|\, j \notin \mathcal{I}(g_{\underline{i},\underline{k}})\}|=n-m-1$, we have
 $J \geq n-m-1$. By definition of $J$, we have $j \in \mathcal{I}(g_{\underline{i},\underline{k}})$ for all $j > J$, and so
 $g_{\underline{i}'',\underline{k}''}^{(J)} \in \CalZ_{n-J}(n-J-1)$.
 All the remaining indices appear in $g_{\ul{i}',\ul{k}'}^{(J)}$, so we have $g_{\underline{i}',\underline{k}'}^{(J)}\in \CalZ_J(m-n+J+1)$.
  We deduce that there exists a bijection between $\CalZ_n(m)$ and the set
\[
 \bigsqcup_{j=n-m-1}^{n-1}\CalZ_j(m-n+j+1)\times \CalZ_{n-j}(n-j-1).
\]
We obtain (\ref{ind2}) by taking the cardinalities of the above sets.

We will  prove Equation (\ref{lem2}) by induction on $n$. For $n=1$, we have $Z_1(0)=1=C_0$. Now assume that  $Z_j(j-1)=C_{j-1}$  for all $1 \leq j \leq n-1$.
By (\ref{ind2}), we have
$$
Z_n(m)=\sum_{j=n-m-1}^{n-1} Z_j(m-n+j+1)Z_{n-j}(n-j-1)=\sum_{j=n-m-1}^{n-1} Z_j(m-n+j+1)C_{n-j-1}$$ 
for all  $0\leq m\leq n-2$.
Then, by (\ref{ind1}), we have 
\[
C_n= \sum_{m=0}^{n-1}Z_n(m)= Z_{n}(n-1)+\sum_{m=0}^{n-2}Z_n(m)=Z_{n}(n-1)+\sum_{m=0}^{n-2}\sum_{j=n-m-1}^{n-1} Z_j(m-n+j+1)C_{n-j-1} .
\]
We change the index $m$ by setting $l:=m-n+j+1$ and observe the equivalence of conditions:
\begin{equation*}
\begin{cases}
0\leq m\leq n-2\\
n-m-1\leq j \leq n-1
\end{cases} \Longleftrightarrow\,\,\,\,\,\,\, \begin{cases}
0\leq l\leq j-1\\
1\leq j \leq n-1
\end{cases}.
\end{equation*}
We obtain
\[
C_n=Z_n(n-1)+\sum_{j=1}^{n-1}\sum_{l=0}^{j-1} Z_j(l)C_{n-j-1}=Z_n(n-1)+\sum_{j=1}^{n-1}C_{n-j-1}\sum_{l=0}^{j-1} Z_j(l). 
\]
Applying again (\ref{ind1}) yields:
\[
C_n=Z_n(n-1)+\sum_{j=1}^{n-1}C_{n-j-1}C_j.
\]
Now, we have the following well-known induction formula for Catalan numbers:
\begin{equation}\label{induction formula}
C_n=\sum_{j=0}^{n-1}C_{n-j-1}C_j.
\end{equation}
Hence, we obtain
\[
C_n=Z_n(n-1)+C_n-C_{n-1},
\]
whence $$Z_n(n-1)=C_{n-1}.$$

Finally, we will prove (\ref{lem3}) also by induction on $n$. For $n=1$,  (\ref{lem3})  becomes $2Z_1(0)=2C_1$, which is true. Now assume that  (\ref{lem3}) holds for $1,2,\ldots,n-1$. By (\ref{ind2}) and (\ref{lem2}), we have:
\begin{align*}
\sum_{m=0}^{n-1} 2^{n-m}Z_n(m)&= 2Z_n(n-1)+\sum_{m=0}^{n-2}2^{n-m}\left(\sum_{j=n-m-1}^{n-1} Z_j(m-n+j+1)Z_{n-j}(n-j-1) \right)\\
&=2C_{n-1}+\sum_{m=0}^{n-2}2^{n-m}\left(\sum_{j=n-m-1}^{n-1} Z_j(m-n+j+1)C_{n-j-1} \right).
\end{align*}
We change the index $m$  by setting  $l:=m-n+j+1$ and swap the sums as before:
\begin{equation*}
\sum_{m=0}^{n-1} 2^{n-m}Z_n(m)=2C_{n-1}+\sum_{j=1}^{n-1}C_{n-j-1}\sum_{l=0}^{j-1}2^{j-l+1} Z_j(l).
\end{equation*}
The induction hypothesis yields:
\[
\sum_{m=0}^{n-1} 2^{n-m}Z_n(m)=2C_{n-1}+\sum_{j=1}^{n-1}C_{n-j-1}2(j+1)C_j=2\sum_{j=0}^{n-1}(j+1)C_{n-j-1}C_j.
\]
Now observe that by using the symmetry $j\leftrightarrow n-j-1=j'$ and then taking the half-sum, we obtain:
\begin{equation*}
\sum_{j=0}^{n-1}(j+1)C_{n-j-1}C_j=\sum_{j'=0}^{n-1}(n-j')C_{n-j'-1}C_{j'}=\frac{n+1}{2}\sum_{j=0}^{n-1}C_{n-j-1}C_j,
\end{equation*}
which is equal to $\frac{n+1}{2}C_n$ by (\ref{induction formula}). Hence, we obtain
$$
\sum_{m=0}^{n-1} 2^{n-m}Z_n(m)=(n+1)C_n.
$$
\end{proof}

We will use the above lemma to calculate the cardinality of the spanning set 
\begin{equation}\label{S_{d,n}}
\boldsymbol{S}_{d,n}=\left\{ \overline{b}_{\ul{i},\ul{k}}\,g_{\ul{i},\ul{k}} \,\left|\, (\ul{i},\ul{k}) \in \SS_n, \ b_{\ul{i},\ul{k}}  \in \mathcal{B}_{d,n}(g_{\ul{i},\ul{k}})\right. \right\}
\end{equation}
for $\YTL$,
where $\mathcal{B}_{d,n}(1)$ is given by Proposition \ref{R'(1)}, while $\mathcal{B}_{d,n}(g_{\ul{i},\ul{k}})$ for $(\ul{i},\ul{k})\in \SS_n \setminus \{(\emptyset,\emptyset)\}$ is defined in Proposition \ref{spanning}. Equivalently, we can write
\begin{equation}\label{S_{d,n}2}
\boldsymbol{S}_{d,n}=\left\{ t_1^{r_1}\ldots t_n^{r_n}\,g_{\ul{i},\ul{k}} \,\left|\, (\ul{i},\ul{k}) \in \SS_n, \ (r_1,\ldots,r_n)  \in \mathcal{E}_{d,n}(g_{\ul{i},\ul{k}})\right. \right\},
\end{equation}
where $\mathcal{E}_{d,n}(g_{\ul{i},\ul{k}})$ is the subset of $\{0,\ldots,d-1\}^n$ defined by:
\begin{equation}\label{p}
(r_1,\ldots,r_n) \in \mathcal{E}_{d,n}(g_{\ul{i},\ul{k}}) \Leftrightarrow  x_1^{r_1}\ldots x_n^{r_n} \in \mathcal{B}_{d,n}(g_{\ul{i},\ul{k}})\ .
\end{equation}

\begin{prop}\label{correct number}
 Let $n \geq 3$. We have 
 $$|\boldsymbol{S}_{d,n}| = {\rm dim}_{\mathbb{C}(u)}(\C(u){\rm YTL}_{d,n}(u)).$$
 \end{prop}
\begin{proof} 
First recall that by (\ref{dim}),
$${\rm dim}_{\mathbb{C}(u)}(\C(u){\rm YTL}_{d,n}(u))=\frac{d\,(nd-n+d+1)}{2}\,C_n -d\,(d-1).$$
By Proposition \ref{R'(1)}, we have
$$|\mathcal{B}_{d,n}(1)|=(2^{n-1}-1)d^2-(2^{n-1}-2)d=2^{n-1}d(d-1) +d -d(d-1).$$
For $(\ul{i},\ul{k})\in \SS_n \setminus \{(\emptyset,\emptyset)\}$, by Corollary \ref{dimcor}, we have
 $$ | \mathcal{B}_{d,n}(g_{\underline{i},\underline{k}})|=  2^{n-m-1}d^2-(2^{n-m-1}-1)d=2^{n-m-1}d(d-1)+d,$$ where
$m=wt(g_{\underline{i},\underline{k}})$. 
Thus, the number of elements of our spanning set is:
\begin{align*}
|\boldsymbol{S}_{d,n}|&=2^{n-1}d(d-1) +d -d(d-1)+\sum_{m=1}^{n-1}\left(2^{n-m-1}d(d-1)+d\right)Z_n(m)\\ 
&=-d(d-1)+d(d-1)\sum_{m=0}^{n-1}2^{n-m-1}Z_n(m)+d\sum_{m=0}^{n-1}Z_n(m).
\end{align*}
Applying (\ref{lem3}) and (\ref{ind1}) to the equality above yields:
$$|\boldsymbol{S}_{d,n}|=-d(d-1)+d(d-1)\frac{n+1}{2}C_n+dC_n={\rm dim}_{\mathbb{C}(u)}(\C(u){\rm YTL}_{d,n}(u)).$$
 \end{proof}

 We can now deduce the main result of this section.
 
 \begin{thm}\label{thesecondtheorem}
 Let $n \geq 3$. The spanning set $\boldsymbol{S}_{d,n}$ is a basis of $\YTL$.
 \end{thm}
 
 \begin{proof}
 Due to Proposition \ref{correct number}, the set $\boldsymbol{S}_{d,n}$ is a basis of $\C(u)\YTL$.  Hence, the elements of $\boldsymbol{S}_{d,n}$ are linearly independent over $\C(u)$, and thus over $\C[u,u^{-1}]$. We deduce that $\boldsymbol{S}_{d,n}$ is a basis of $\YTL$.
 \end{proof}

 \begin{cor} 
 Let $n \geq 3$ and $(\ul{i},\ul{k}) \in \SS_n$. We have $R'(g_{\underline{i},\underline{k}}) = R(g_{\underline{i},\underline{k}})$ and the set
 $$\{b_{\ul{i},\ul{k}} + R(g_{\ul{i},\ul{k}})\,|\, b_{\ul{i},\ul{k}}  \in \mathcal{B}_{d,n}(g_{\ul{i},\ul{k}})\}$$
is a basis of $A_{d,n}/R(g_{\underline{i},\underline{k}})$.
 \end{cor}


\begin{thebibliography}{00}

\bibitem[ChPo]{ChPdA} M.~Chlouveraki, L.~Poulain d'Andecy, {\em Representation theory of the Yokonuma--Hecke algebra}, Adv.~Math.~{\bf 259} (2014), 134--172.

\bibitem[DaKo]{DaKo} V.~Danilov, G.~Koshevoy, {\em A simple proof of associativity and commutativity of LR-coefficients (or the hive ring)}, arXiv:math/0409447.

\bibitem[GePf]{gp} M.~Geck, G.~Pfeiffer, Characters of finite Coxeter groups and Iwahori-Hecke algebras, London Math. Soc. Monographs, New Series 21, Oxford University Press, New York, 2000.

\bibitem[GJKL]{gjkl} D.~Goundaroulis, J.~Juyumaya,  A.~Kontogeorgis, S.~Lambropoulou, {\em The Yokonuma--Temperley--Lieb algebra}, arXiv:1012.1557.

\bibitem[Jo1]{jo2} V.~F.~R.~Jones, {\em  Index for Subfactors}, Invent.~Math.~{\bf 72} (1983), 1--25.

\bibitem[Jo2]{jo} V.~F.~R.~Jones, {\em Hecke algebra representations of braid groups and link polynomials}, Annals of Math.~{\bf 126} (1987),  no. 2, 335--388.
   
\bibitem[Ju1]{ju1} J.~Juyumaya, {\em Sur les nouveaux g\'en\'erateurs de l'alg\`ebre de Hecke H(G,U,1)}, J.~Algebra {\bf 204} (1998) 49--68.

\bibitem[Ju2]{ju} J.~Juyumaya, {\em Markov trace on the Yokonuma--Hecke algebra}, J.~Knot Theory Ramifications {\bf 13} (2004) 25--39.
   
\bibitem[JuKa]{juka} J.~Juyumaya, S. Kannan, {\em Braid relations in the Yokonuma--Hecke algebra}, J.~Algebra {\bf 239} (2001) 272--297.      

\bibitem[JuLa1]{jula1} J.~Juyumaya, S.~Lambropoulou, {\em $p$-adic framed braids}, Topology Appl.~{\bf 154} (2007) 1804--1826.

\bibitem[JuLa2]{jula2} J.~Juyumaya, S.~Lambropoulou, {\em $p$-adic  framed braids II}, Adv.~Math.~{\bf 234} (2013) 149--191.

\bibitem[JuLa3]{jula3} J.~Juyumaya, S.~Lambropoulou, {\em An adelic extension of the Jones polynomial}, M. Banagl, D. Vogel (eds.) The mathematics of knots, Contributions
in the Mathematical and Computational Sciences, Vol.~1, Springer.
   
\bibitem[JuLa4]{jula4} J.~Juyumaya, S.~Lambropoulou, {\em An invariant for singular knots}, J.~Knot Theory Ramifications {\bf 18}(6) (2009) 825--840.

\bibitem[JuLa5]{jula5} J.~Juyumaya, S.~Lambropoulou, {\em On the framization of knot algebras}, to appear in New Ideas in Low-dimensional Topology, L. Kauffman, V. Manturov (eds), Series of Knots and Everything, WS.
   
 \bibitem[Le]{Le} M.~van Leeuwen, {\em The Littlewood--Richardson rule, and related combinatorics}, arXiv:math/9908099.

\bibitem[Mac]{Mac} I.~G.~Macdonald, Symmetric Functions and Hall Polinomials, Oxford Mathematical Monographs, Clarendon Press, Oxford 1979.

\bibitem[Sp]{Spe} W.~Specht, Fine Verallgemeinerung der symmetrischen Gruppe, Schriften Math.~Seminar (Berlin) {\bf 1} (1932) 1--32.

\bibitem[St]{St} J.~R.~Stembridge, {\em On the eigenvalues of representations of reflection groups and wreath products}, Pacific.~J.~Math.~{\bf 140}(2) (1989) 353--396.

\bibitem[Th1]{thi} N.~Thiem, {\em Unipotent Hecke algebras: the structure, representation theory, and combinatorics}, Ph.D. Thesis, University of Wisconsin (2004).

\bibitem[Th2]{thi2} N.~Thiem, {\em Unipotent Hecke algebras of \,${\rm GL}_n(\mathbb{F}_q)$}, J.~Algebra {\bf 284} (2005) 559--577. 

\bibitem[Th3]{thi3} N.~Thiem, {\em A skein-like multiplication algorithm for unipotent Hecke algebras}, Trans.~Amer.~Math.~Soc.~{\bf 359}(4) (2007) 1685--1724.

\bibitem[Yo]{yo} T.~Yokonuma, {\em Sur la structure des anneaux de Hecke d'un groupe de Chevalley fini}, C.~R.~Acad.~Sci.~Paris Ser.~I Math.~{\bf 264}  (1967)  344--347.

\end{thebibliography}
\end{document}